\pdfoutput=1
\RequirePackage{ifpdf}
\ifpdf 
\documentclass[pdftex]{sigma}
\else
\documentclass{sigma}
\fi

\numberwithin{equation}{section}

\newtheorem{theo}{Theorem}[section]
\newtheorem*{Theorem*}{Theorem}
\newtheorem{coro}[theo]{Corollary}
\newtheorem{lemm}[theo]{Lemma}
\newtheorem{prop}[theo]{Proposition}

\theoremstyle{definition}
\newtheorem{defi}[theo]{Definition}

\newtheorem{exam}[theo]{Example}
\newtheorem{rema}[theo]{Remark}

\usepackage{tikz}
\usetikzlibrary{decorations.markings}

\usepackage{tikz-cd}

\newcommand{\supp}{\mathrm{supp}}

\renewcommand{\ker}{\mathrm{ker}}

\newcommand{\End}{\mathrm{End}}
\newcommand{\Hom}{\mathrm{Hom}}

\newcommand{\Res}{\mathrm{Res}}
\newcommand{\pardeg}{\mathrm{pardeg}}

\newcommand{\disc}{\mathrm{disc}}

\newcommand{\reg}{\mathrm{reg}}
\newcommand{\ord}{\mathrm{ord}}
\renewcommand{\div}{\mathrm{div}}
\newcommand{\Div}{\mathrm{Div}}

\newcommand{\Jac}{\mathrm{Jac}}

\newcommand{\codim}{\mathrm{codim}}

\newcommand{\toy}{\mathrm{toy}}

\renewcommand{\and}{\mathrm{and}}

\newcommand{\rest}{\: \rule[-3.5pt]{0.5pt}{11.5pt}\,{}}
\renewcommand{\to}{\rightarrow}
\newcommand{\wh}{\widehat}
\newcommand{\ualpha}{\underline{\alpha}}

\newcommand{\calB}{\mathcal{B}}

\newcommand{\calE}{\mathcal{E}}
\newcommand{\calF}{\mathcal{F}}

\newcommand{\calI}{\mathcal{I}}

\newcommand{\calL}{\mathcal{L}}
\newcommand{\calM}{\mathcal{M}}
\newcommand{\calN}{\mathcal{N}}
\newcommand{\calO}{\mathcal{O}}
\newcommand{\calP}{\mathcal{P}}

\newcommand{\ZZ}{\mathbb{Z}}
\newcommand{\Z}{\mathbb{Z}}

\newcommand{\RR}{\mathbb{R}}

\newcommand{\CC}{\mathbb{C}}
\newcommand{\C}{\mathbb{C}}
\newcommand{\PP}{\mathbb{P}}

\newcommand{\id}{\mathrm{id}}
\newcommand{\tr}{\mathrm{tr}}
\newcommand{\diag}{\mathrm{diag}}

\renewcommand{\d}{\mathrm{d}}

\DeclareMathOperator{\SFM}{SFM}
\DeclareMathOperator{\FM}{FM}

\newcommand{\Prym}{\mathrm{Prym}}
\newcommand{\Nm}{\mathrm{Nm}}

\newcommand{\Hit}{\mathrm{Hit}}

\newcommand{\GL}{\mathrm{GL}}
\newcommand{\SL}{\mathrm{SL}}
\newcommand{\Sp}{\mathrm{Sp}}

\newcommand{\PSL}{\mathrm{PSL}}
\newcommand{\PGL}{\mathrm{PGL}}

\newcommand{\gl}{\mathfrak{gl}}

\begin{document}

\allowdisplaybreaks

\newcommand{\arXivNumber}{2310.15716}

\renewcommand{\PaperNumber}{022}

\FirstPageHeading

\ShortArticleName{Visible Lagrangians for Hitchin Systems and Pillowcase Covers}

\ArticleName{Visible Lagrangians for Hitchin Systems\\ and Pillowcase Covers}

\Author{Johannes HORN and Johannes SCHWAB}

\AuthorNameForHeading{J.~Horn and J.~Schwab}

\Address{Goethe-Universit\"at Frankfurt am Main, Institut f\"ur Mathematik,\\ Robert-Mayer-Str. 6-8, 71732 Frankfurt, Germany}
\Email{\mail{horn@math.uni-frankfurt.de}, \mail{schwab@math.uni-frankfurt.de}}
\URLaddress{\url{https://www.uni-frankfurt.de/115632224/Dr__Johannes_Horn}, \newline
\hspace*{10.5mm}\url{https://www.uni-frankfurt.de/115634639/Johannes_Schwab}}

\ArticleDates{Received September 18, 2025, in final form February 18, 2026; Published online March 09, 2026}

\Abstract{We study complex Lagrangians in Hitchin systems that factor through a proper subvariety of the Hitchin base non-trivially intersecting the regular locus. This gives a~general framework for several examples in the literature. We compute the fiber-wise Fourier--Mukai transform of flat line bundles on visible Lagrangians. This proposes a construction of mirror dual branes to visible Lagrangians. Finally, we study a new example of visible Lagrangians in detail. Such visible Lagrangian exists whenever the underlying Riemann surface is a pillowcase cover. The proposed mirror dual brane turns out to be closely related to Hausel's toy model.}


\Classification{14H60; 30F30; 14H70}

\section{Introduction}
Mirror symmetry questions about Higgs bundle moduli spaces have been intensively studied in recent years. The work of Hausel and Thaddeus \cite{HauselThaddeus} initiated the direction of research by observing the SYZ mirror symmetry of the $\SL(n,\CC)$- and $\PGL(n,\CC)$-Hitchin system and proving topological mirror symmetry for $n=2,3$. Later Donagi and Pantev \cite{DonagiPantev} established the duality between Hitchin systems associated to a complex reductive Lie group $G$ and its Langlands dual group $G^{\rm L}$. To a complex reductive group $G$ and a closed Riemann surface~$X$, we associate a~moduli space of $G$-Higgs bundles $\calM_G$ with a Hitchin map $\Hit\colon \calM_G \to \calB_G$ to a half-dimensional vector space. Then Donagi and Pantev showed that there is an isomorphism $\calB_G \cong \calB_{G^{\rm L}}$, such that the generic fibers over corresponding points under this isomorphism are torsors over dual abelian varieties. Furthermore, the Fourier--Mukai transform yields an equivalence of derived categories of the regular loci of the $G$ and $G^{\rm L}$-Hitchin system.

About the same time the work of Kapustin and Witten \cite{KapustinWitten} raised the question about mirror symmetry of special subvarieties referred to as branes in the physical literature. A~brane is a~pair $(\calN,F)$ of a subvariety $\calN$ and a sheaf $F$ supported on $\calN$ with special geometric properties. This initiated plenty of mathematical research to find examples of branes or their supports for $G$\nobreakdash-Hitchin system \cite{BaragliaSchaposnik1,BaragliaSchaposnik2, FrancoJardim, FrancoPeon, HauselHitchin, HellerSchaposnik, HitchinSpUmm, HitchinSpinors}. However, \cite{KapustinWitten} also propose a correspondence between branes under mirror symmetry. This seems to be less considered in the mathematical literature (see \cite{FrancoJardim, FrancoPeon, HauselHitchin, HitchinSpUmm} for exceptions).

In this paper, we consider so-called $(B,A,A)$-branes, that is pairs $(\calN,F)$, where $\calN$ is a~complex Lagrangian subvariety and $F$ is a flat bundle on $\calN$. We also describe the subvarieties and sheaves related to these $(B,A,A)$-branes by the Fourier--Mukai transform. The work of Kapustin and Witten \cite{KapustinWitten} suggests that these are $(B,B,B)$-branes, i.e., hyperholomorphic subvarieties with a hyperholomorphic sheaf. We give indication for this conjecture for our main example.

More specifically, we are interested in complex Lagrangians $\calL$ such that the restriction of the Hitchin map factors through a proper subvariety $\calB'=\Hit(\calL)\subsetneq \calB$. Such Lagrangians are called \textit{visible} in the symplectic geometry literature \cite{visible}. This is complementary to the recent work of Hausel and Hitchin~\cite{HauselHitchin}, who studied the upward flow to certain points in the nilpotent cone.

We first abstractly consider visible Lagrangians in the $G=\GL(n,\CC)$ and $G=\SL(n,\CC)$-Hitchin system and describe their proposed mirror dual by computing the Fourier--Mukai transform of flat sheaves on them. The first main result is the following.

\begin{theo}[{Theorem~\ref{theo:FM}}]\label{theo:1}
 Let $\calL \subset \calM_G$ be a visible Lagrangian, such that $\calB'=\Hit(\calL) \subsetneq \calB_G$ and $\calB' \cap \calB^\reg_G \neq \varnothing$. Let $s\colon \calB' \to \calL$ be a section of $\Hit\rest_{\calL}$. The fiber-wise Fourier--Mukai transform of the structure sheaf $\calO_\calL$ is supported on a holomorphic symplectic subvariety $\calI_{\calL,s} \subset \calM_{G^{\rm L}}$ such that $\Hit\rest_{\calI}\colon \calI_{\calL,s} \to \calB'$ is an algebraically completely integrable system.
\end{theo}

Note that every hyperholomorphic subvariety is holomorphic symplectic. Hence, this observation fits well with the mirror symmetry proposal in~\cite{KapustinWitten}.

On the question of existence, we mention three situations under which we expect visible Lagrangians, examples of which appeared in the literature. The first is Lie-theoretic. An~inclusion of a semisimple Lie group $G_1$ into a reductive Lie group $G_2$ defines a morphism of Higgs bundle moduli spaces $\calM_{G_1} \to \calM_{G_2}$. Its image is a hyperholomorphic subvariety. We expect the mirror dual to this hyperholomorphic subvarieties to be visible Lagrangians. We consider the example of $\SL(n,\CC) \subset \GL(n,\CC)$ in Section~\ref{sec:visi_Lagra_Hitchin}. Another example where $G_1=\Sp(2n,\CC)$ and $G_2=\GL(2n,\CC)$ was considered by Hitchin in \cite{HitchinSpUmm}. Other than that, we hope to return to this type of visible Lagrangians in a subsequent work and will mainly focus on $G=\SL(2,\CC)$ in the remainder of the paper.
The second type of visible Lagrangians appeared in the work of \cite{FGOP, FHHO, FrancoPeon}. They are completely contained in the singular locus of the Hitchin map and substantially use the geometry of the singular Hitchin fibers. In particular, they do not fall within the scope of this work.

The focus of this work is a third type of visible Lagrangians related to the symmetries of the underlying Riemann surface. The general fiber of the Hitchin system is a torsor over an abelian variety. A necessary condition for the existence of a visible Lagrangian $\calL \to \calB'$ is that the Hitchin fibers over $\calB'$ correspond to reducible (i.e., non-simple) abelian varieties. Comparing the dimension of the $\SL(2,\CC)$-Hitchin base and the reducible locus in the corresponding moduli space of abelian varieties suggests that there are finitely many directions in the Hitchin base, where the Hitchin fiber is isomorphic to a reducible abelian variety. Hence, it is natural to look for visible Lagrangians $\calL \to \calB'$ over lines $\CC \underline{a} \subset \calB_G$ in the Hitchin base. We have the following second main theorem, of which a $\SL(n,\CC)$-version is proven in Theorem~\ref{theo:visi_Lagr_over_line}.

\begin{theo}[Corollary \ref{coro:visible_Lagr_over_line_sl2C}]\label{theo:intro:1-dim_examples}
 Let $q \in H^0\bigl(X,K_X^2\bigr)$ be a quadratic differential with simple zeros only. Then there exists a visible Lagrangian
 \[ \calL \to \calB'=\{t q \mid t \in \CC\} \subset \calB_{\SL(2,\CC)}(X)
 \]
 if and only if $(X,q)$ is a pillowcase cover.
\end{theo}

The notion of pillowcase cover stems from the theory of flat surfaces. It means that there is a~covering $X \to \PP^1$, such that the quadratic differential $q$ is the pullback of a quadratic differential on $\PP^1$ with four simple poles. The later should be figured as a pillowcase, see Figure~\ref{fig:Pillowcase}. We give a short introduction to the idea of flat surfaces in Section~\ref{sec:rhomb-tiled-surf}.

Motivated by the above considerations on the moduli space of abelian varieties, in Section~\ref{sec:pill-sever-ways}, we study examples of Riemann surfaces, where there exist several lines in the $\SL(2,\CC)$-Hitchin base associated to visible Lagrangians. We prove the following result, which might be of independent interest from the point of view of flat surfaces.

\begin{theo}[{Proposition~\ref{prop:non-isomorphic_pillow}}]\label{theo:intro_3}
 There exist infinitely many genera $g$, such that there exists a~Riemann surface~$X$ of genus~$g$ with two quadratic differentials $q_1$, $q_2$ with simple zeros only, such that $(X,q_i)$ are pillowcase covers and $q_1$, $q_2$ are not related by pullback along an automorphism of~$X$.
\end{theo}

Finally, we consider the subintegrable system $\calI \subset \calM_{\PGL(2,\CC)}$ of Theorem~\ref{theo:1} associated to the visible Lagrangian $\calL$ of Theorem~\ref{theo:intro:1-dim_examples}. We observe that $\calI$ is birational to Hausel's toy model~\cite{HauselToy}. Under the natural extra condition on the pillowcase cover to be \textit{uniform}, we prove the following theorem that confirms the Kapustin--Witten picture for visible Lagrangians of this kind. All the pillowcase covers of Theorem~\ref{theo:intro_3} are uniform.

\begin{theo}[{Corollary \ref{coro:hyperholomorphic}}]
 Let $(X,q)$ be a uniform pillowcase cover with simple zeros only and $\calL \to \CC q$ the visible Lagrangian of Theorem~{\rm \ref{theo:intro:1-dim_examples}}. Then the associated subintegrable system $\calI \subset \calM_{\PGL(2,\CC)}$ of Theorem~{\rm \ref{theo:1}} is a hyperholomorphic subvariety.
\end{theo}

\section{Symplectic geometry}
\subsection{Completely integrable systems}
\begin{defi}\label{defi:integrable_system} A completely integrable system is a holomorphic symplectic manifold $(M,\Omega)$ together with a proper flat morphism $H\colon M \to B$ to a complex manifold $B$, such that on the complement $B\setminus S$ of some proper closed subvariety $S$ the fibers of $H$ are complex Lagrangian tori. It~is called algebraically completely integrable, if the Lagrangian tori are endowed with continuously varying polarizations $\rho_b \in H^{(1,1)}(M_b) \cap H^2(M_b,\ZZ)$, i.e., they are abelian varieties.
\end{defi}

We will refer to $B^\reg=B\setminus S$ as the regular locus and to $S$ as the singular locus of a completely integrable system.
\begin{defi} An integral affine structure on a smooth manifold $B$ is a torsion-free flat connection $\nabla$ on the tangent bundle $TB$ together with a $\nabla$-covariant sublattice $\Lambda \subset TB$ of maximal rank. A~submanifold $B' \subset B$ is called rational with respect to the integral affine structure, if for all $b \in B'$ the tangent space $T_bB'$ is generated by a rational linear combination of elements of~$\Lambda_b$.
\end{defi}
For completely integrable systems, there is a natural identification of the cotangent bundle to the base with the torus-invariant vector fields along to the fibers of $H$ given as follows. Let $\alpha \in T_b^\vee B$ be a (holomorphic) one-form then there exists an invariant vector field $X$ on $M_b$, such that $\Omega(X,\cdot)=H^*\alpha$. Denoting by $VM$ the bundle of invariant vector fields along to the fibers over $B^\reg$ we obtain an identification $T^\vee B^\reg \cong VM$. Locally over $U \subset B^\reg$ we can choose a section of $H\colon M \to B$ and identify $H^{-1}(U)\cong V_UM/\Lambda$ for a family of lattices $\Lambda \subset V_UM$. The above isomorphism defines a family of lattices $\Lambda \subset T^\vee B^\reg$. The dual family of lattices $\Lambda^\vee\subset TB^\reg$ defines a torsion-free flat connection on $T B^\reg$, where a section is flat if and only if it is constant with respect to lattice coordinates (see \cite[Section~3]{Freed} for more details). This defines an integral affine structure on the base of an completely integrable system.

\subsection{Visible Lagrangians}\label{sec:visible}
This idea goes back to lecture notes of Jonathan David Evans \cite[Chapter~5]{visible} in the context of real completely integrable systems.

\begin{defi}[visible Lagrangians]\label{def:visible-lagrangians}
Let $H\colon M \to B$ be a completely integrable system. A Lagrangian subvariety $\calL \subset M$ is called \textit{visible}, if it is closed and $H\rest_{\calL}\colon \calL \to B$ factors as $H\rest_{\calL}=f \circ g$, where $f\colon B' \to B$ is an embedding of a proper subvariety $B'$ such that $g\rest_{B'\setminus S'}\colon \calL\rest_{B'\setminus S'} \to B'\setminus S'$ is a smooth fiber bundle on the complement of some proper subvariety $S' \subsetneq B'$.
\end{defi}

The simplest example of a visible Lagrangian is a complex torus fiber $M_b$ with $B'=\{ b \}$. On the other hand, a Lagrangian section $s\colon B \to M$ is an example of a Lagrangian that is not visible. We denote $B'^{\reg}=B'\setminus S'$. For the visible Lagrangians considered in the present work we will mostly have $B'^\reg=B' \cap B^\reg$.

\begin{theo}\label{theo:visible_Lagr} Let $H\colon M \to B$ be a completely integrable system and $\calL \to B' \subsetneq B$ a visible Lagrangian with $B'^\reg \subset B^\reg$. Then at each smooth point $b \in B'^\reg$ the subvariety $B'$ is rational with respect to the integral affine structure on $B^\reg$ and for $b \in B'^\reg$ the fiber $\calL_b$ is a union of complex tori generated by the invariant vector fields $T_bB'^\perp \subset V_bM$.
\end{theo}
\begin{proof} Locally at a smooth point $b \in B'^\reg$ the subvariety $B' \subset B$ is cut out by $k=\codim B'$ many functions $f_1,\dots,f_k\in \calO_B$. We can associate invariant vector fields $X_i \in VM\rest_{B'}$ along the torus fibers so that $\Omega(X_i, \cdot)= H^* \mathrm{d} f_i$ for $i=1,\dots,k$. Let $m \in \calL_b$ and $Y \in T_mM$, such that $DH(Y) \in T_bB'$, then
 \begin{align} \Omega(X_i,Y)=H^*\mathrm{d}f_i Y=\mathrm{d}f_i(DH(Y)) = 0. \label{eq:horizontal_vf}
 \end{align}
 Therefore, the connected components of the fiber of $\calL$ over $b \in B'^\reg$ are integral submanifolds of the distribution $V\calL=\mathrm{span}(X_1, \dots, X_k)$. Hence the connected components of the fibers of $\calL_b \to B'^\reg$ are complex subtori. In particular, the subspace $V_b\calL \subset V_bM$ is rational with respect to the lattice $\Lambda_b$. By definition, this is equivalent to $H_*T\calL=TB'$ being rational with respect to the integral affine structure on $B$. Finally, by (\ref{eq:horizontal_vf}) we have $ V\calL \subset (TB'^\reg)^\perp \subset T^\vee B^\reg \cong VM$. The first inclusion is an equality both being of rank $k$.
\end{proof}

\section{Hitchin systems \label{sec:hitchin-systems}}
In this section, we will briefly review Hitchin systems - the algebraically completely integrable systems of interest in this work. Then we will give an example of a visible Lagrangian of $\calM_{\GL(n,\CC)}$ that stems from the embedding of $\SL(n,\CC) \subset \GL(n,\CC)$. In the remainder of the paper we will focus on visible Lagrangians that do not come from Lie theory.

\subsection[Preliminaries about the GL(n,C) and SL(n,C)-Hitchin systems]{Preliminaries about the $\boldsymbol{\GL(n,\CC)}$ and $\boldsymbol{\SL(n,\CC)}$-Hitchin systems}\label{sec:prel-Hitchin_syst}

Let $G$ be a complex reductive Lie group and $X$ a Riemann surface, then there is a moduli space of stable $G$-Higgs bundles. More precisely, we denote by $\calM_G(X)$ the neutral component of the moduli space of stable $G$-Higgs bundles. It is a hyperk\"ahler manifold, in particular, holomorphic symplectic. There is the Hitchin map $\Hit\colon \calM_G(X) \to \calB_G(X)$ to a half-dimensional vector space $\calB_G(X)$, which is an algebraically completely integrable system in the sense of Definition~\ref{defi:integrable_system}. So it is sensible to ask for the existence of visible Lagrangians.

Let us be more concrete about the cases $G=\GL(n,\CC)$ and $G=\SL(n,\CC)$. A $\GL(n,\CC)$-Higgs bundles is a pair $(\calE,\Phi)$ of a holomorphic vector bundle $\calE$ together with a section $\Phi \in H^0(X,\End(\calE)\otimes K_X)$, where $K_X$ is the canonical bundle of~$X$. Considering the neutral component of the $\GL(n,\C)$-moduli space means to fix the degree of $\calE$ to be $0$. For $(\calE,\Phi)$ to be a~$\SL(n,\CC)$-Higgs bundle, we add the condition of the determinant bundle of $\calE$ and trace of $\Phi$ to be trivial. In the $\GL(n,\CC)$-case, the Hitchin map is given by
\[ \Hit\colon \ \calM_{\GL(n,\CC)} \to \calB_{\GL(n,\CC)}= \bigoplus_{i=1}^n H^0\bigl(X,K^i\bigr), \qquad (\calE,\Phi) \mapsto \underline{a}(\Phi)=(a_1(\Phi),\dots,a_n(\Phi)),
\]
where $a_i \in \CC[\gl(n,\CC)]^{\GL(n,\CC)}$ is the $i$-th coefficient of the characteristic polynomial. In the $\SL(n,\CC)$-case, it is given by
\[ \Hit \colon \ \calM_{\SL(n,\CC)} \to \calB_{\SL(n,\CC)}= \bigoplus_{i=2}^n H^0\bigl(X,K^i\bigr), \qquad (\calE,\Phi) \mapsto \underline{a}(\Phi)=(a_2(\Phi),\dots,a_n(\Phi)).
\]
Fixing a point $\underline{a} \in \calB_G$ in the Hitchin base, the eigenvalues of $\Phi$ define a branched $n$-sheeted cover $\pi\colon \Sigma_{\underline{a}} \to X$ -- the so-called spectral curve. The discriminant of the characteristic polynomial defines a map \smash{$\disc_G\colon \calB_G \to H^0\bigl(X,K^{r(r-1)}\bigr)$}. The discriminant locus $\Delta_G \subset \calB_G$ is the preimage of the sections of \smash{$H^0\bigl(X,K^{r(r-1)}\bigr)$} with higher order zeros under $\disc_G$. Its complement $\calB_G^{\reg}=\calB_G \setminus \Delta_G$ is referred to as the regular locus. In particular, for $\underline{a} \in \calB^{\reg}_G$ the spectral curve~$\Sigma_{\underline{a}}$ is smooth. For $G=\SL(2,\CC)$, the regular locus \smash{$\calB_{\SL(2,\CC)}^{\reg}$} is the locus of quadratic differentials with simple zeros only.

The fibers of the $\GL(n,\CC)$-Hitchin map over \smash{$\calB_{\GL(n,\CC)}^{\reg}$} -- the so-called regular fibers -- are torsors over $\Jac(\Sigma_{\underline{a}})$ via the spectral correspondence. The branched cover $\pi\colon \Sigma_{\underline{a}} \to X$ defines a~norm map $\Nm_{\Sigma/X}\colon \Jac(\Sigma_{\underline{a}}) \to \Jac(X)$. The kernel of this morphism defines the Prym variety $\Prym(\Sigma_{\underline{a}})$. A regular fiber of the $\SL(n,\CC)$-Hitchin system is a torsor over $\Prym(\Sigma_{\underline{a}})$. The neutral component of the Higgs bundles moduli space allows for the existence of sections $s_H\colon \calB_G \to \calM_G$ of the Hitchin map -- the so-called Hitchin sections. For $\GL(n,\CC)$, they are given by
\[ s_H(\underline{a})=\left( K^{\frac{n-1}{2}} \oplus \dots \oplus K^{-\frac{n-1}{2}}, \begin{pmatrix} a_1 & a_2 & \dots & a_n \\ 1 & a_1 & \ddots & \vdots \\ & \ddots & \ddots & a_2 \\ &&1& a_1 \end{pmatrix} \right)
\] depending on a choice of a~square-root~$K^{\frac{1}{2}}$.

The tangent space to $\Jac(\Sigma_{\underline{a}})$ at the identity is $H^1(\Sigma_{\underline{a}},\calO_{\Sigma_{\underline{a}}})$ by the exponential sequence. The inclusion of the Hitchin fiber into $\calM_{\GL(n,\CC)}$ yields an exact sequence of tangent spaces
\[ 0 \to H^1(\Sigma_{\underline{a}},\calO_{\Sigma_{\underline{a}}}) \to T_{(\calE,\Phi)}\calM_{\GL(n,\CC)} \to H^0(\Sigma_{\underline{a}},K_{\Sigma_{\underline{a}}}) \to 0
\]
(see \cite[Proposition~8.2]{Markman}).
The holomorphic symplectic form identifies the vertical tangent vectors with the dual of the tangent space to the Hitchin base. Combining this with Serre duality yields the following identification of the tangent space to base with differentials on $\Sigma_{\underline{a}}$.

\begin{prop}[{\cite[Proposition~3.4]{Baraglia}}]\label{prop:Baraglia}
 Let \smash{$\underline{a} \in \calB_{\GL(n,\CC)}^\reg$}. The identification of the tangent space $T_{\underline{a}}\calB_{\GL(n,\CC)}$ with the dual of the tangent space to the fiber $T^\vee_L\Jac(\Sigma_{\underline{a}}) \cong H^0(\Sigma_{\underline{a}},K_\Sigma)$ is given by
 \begin{align*}
 t\colon \ \bigoplus\limits_{i=1}^n H^0\bigl(X,K_X^i\bigr) & \to H^0(\Sigma_{\underline{a}},K_{\Sigma_{\underline{a}}})
 \\
\sum\limits_{i=1}^n \alpha_i X_i & \mapsto \frac{1}{s_B} \sum\limits_{i=1}^n \alpha_i \pi^*X_i\bigl(\lambda^{n-i}+\pi^*a_2 \lambda^{n-2-i}+ \dots + \pi^*a_{n-i}\bigr),
 \end{align*}
 where $s_B=d \pi \in H^0\bigl(\Sigma_{\underline{a}},\pi^*K_X^{n-1}\bigr)$ vanishes at the branch divisor.
\end{prop}
\begin{proof}
 The original statement gives the isomorphism with values in $H^0(\Sigma_{\underline{a}},\pi^*K_X^n)$. We have the isomorphism of sheaves $K_{\Sigma_{\underline{a}}} \to \pi^*K_X^n$ given by $ \phi \mapsto \phi s_B\rest_U$ for $\phi \in K_{\Sigma_{\underline{a}}}\rest_U$. This is well defined by the following: Let $(U,w)$ a coordinate disc centered at $y \in \Sigma$, such that $\pi\colon U \to \pi(U)$, $w \mapsto z=w^b$ with $b \geq 1$. Then $s_B\rest_U= b w^{b-1} (\pi^* \d z)^{n-1}$. Hence, $\phi \cdot s_B(U)=f \d w \cdot b w^{b-1} (\pi^* \d z)^{n-1}=f (\pi^* \d z)^{n}$, where we used that $\pi^* \d z= b w^{b-1} \d w$. Composing Baraglia's isomorphism \cite[Proposition~3.4]{Baraglia} with the inverse of the above defines the asserted isomorphism~$t$.\looseness=1
\end{proof}

In the following proposition, we take a algebro-geometric point of view and will consider the family of smooth curves $\Sigma \to \calB_G^\reg$. The $\GL(n,\CC)$-Hitchin system is a torsor over the abelian scheme defined by the relative Jacobian $\Jac\bigl(\Sigma/\calB_G^\reg\bigr)$.

\begin{prop}\label{prop:criterion} A closed subvariety $\calL \subset \calM_{\GL(n,\CC)}$ over a proper subvariety $\calB' \subsetneq \calB_{\GL(n,\CC)}$ with $\calB'^\reg \subset \calB^\reg$ and connected fibers is a visible Lagrangian if and only if
 \begin{itemize}\itemsep=0pt
 \item[$(i)$] There exists an abelian subscheme $A \subset \Jac\bigl(\Sigma/\calB_G^\reg\bigr) \rest_{B'^\reg}$ over $\calB'^\reg$, such that $\calL$ is an $A$-torsor.
 \item[$(ii)$] The relative tangent bundle $T_{A/\calB'^\reg} \subset T_{\Jac(\Sigma)/\calB'^\reg}$ is the kernel of the map
 \[ T_{\Jac(\Sigma)/\calB'^\reg} \cong R^1\pi_*\calO_\Sigma \to \CC^{\dim \calB'^\reg}
 \]
 defined by evaluating on the image of
 \[ T\calB'^\reg \to R^0\pi_*K_\Sigma, \qquad X \mapsto t(X)
 \]
 using the Serre pairing on $\Sigma$. Here $t$ was defined in Proposition~{\rm \ref{prop:Baraglia}}.
 \end{itemize}
\end{prop}

\begin{proof} By Theorem~\ref{theo:visible_Lagr}, fibers of $\calL \to \calB'^\reg$ are complex subtori in the Hitchin fiber. Hence, using a local section $s\colon U \to \calL$ on an open $U \subset \calB'^\reg$, we can identify the fibers with an abelian subscheme of $A \to \Jac(\Sigma/\calB'^\reg)$. Condition (ii) is the family version of Theorem~\ref{theo:visible_Lagr} reformulated by using the observations about the holomorphic symplectic form on $\calM_{\GL(n,\CC)}$ in the previous paragraph. Conversely, condition (ii) is equivalent to the restriction of the symplectic form of~$\calM_{\GL(n,\CC)}$ to be zero on the tangent bundle to $\calL$. Hence, $\calL$ is a Lagrangian.
\end{proof}

In general, a visible Lagrangian might not have connected fibers. We will give an example in Theorem~\ref{theo:Lagr_SL_in_GL_multisection}.

To obtain an analogous statement in the $\SL(n,\CC)$-case, we have to identify the image of \smash{$T\calB^{\reg}_{\SL(n,\CC)} \subset T\calB^{\reg}_{\GL(n,\CC)}$} through the isomorphism $t$ of Proposition~\ref{prop:Baraglia}. The pullback
\[ \pi^*\colon \ H^0(X,K_X) \rightarrow H^0({\Sigma_{\underline{a}}},K_{\Sigma_{\underline{a}}})
\]
defines an inclusion of the differentials on $X$ into the differentials on~${\Sigma_{\underline{a}}}$. We define the linear map
\[ \mathrm{pr}_X\colon \ H^0({\Sigma_{\underline{a}}},K_{\Sigma_{\underline{a}}}) \to H^0(X,K_X), \qquad \lambda \mapsto \eta,
\]
where $\eta$ is define as follows: Let $U \subset X$ be a trivially covered open set, i.e., $\pi^{-1}(U)=\bigsqcup_{i=1}^n U_i$. Define $\eta(U)= \frac1n \sum_{i=1}^n \lambda(U_i)$. Then $\eta$ extends to an abelian differential on $X$ by the Riemann extension theorem. Clearly, $\mathrm{pr}_X \circ \pi^*=\id$. Denote by $H^0({\Sigma_{\underline{a}}},K_{\Sigma_{\underline{a}}})^-$ the kernel of $\mathrm{pr}_X$. Then this induces a splitting $H^0({\Sigma_{\underline{a}}},K_{\Sigma_{\underline{a}}})=H^0(X,K_X) \oplus H^0({\Sigma_{\underline{a}}}, K_{\Sigma_{\underline{a}}})^-$.

The Prym variety is the kernel of the Norm map $\Nm_{\Sigma/X}\colon \Jac(\Sigma_{\underline{a}}) \to \Jac(X)$ induced by the Norm map on structure sheaves $\Nm_{\Sigma/X}\colon \pi_*\calO_{\Sigma_{\underline{a}}} \to \calO_X$, $\pi_*f \to \det(\pi_*f)$. Tangentially, we have a splitting of $\pi_*\calO_{\Sigma_{\underline{a}}}= \calO_X \oplus \pi_*\calO_{\Sigma_{\underline{a}}}^-$, where $\pi_*\calO_{\Sigma_{\underline{a}}}^-$ is the kernel of $\mathrm{nm}_{\Sigma/X}\colon \pi_* \calO_{\Sigma_{\underline{a}}} \to \calO_X, \pi_*f \mapsto \tr(\pi_*f)$. This yields a splitting of the tangent space to the $\GL(n,\CC)$-Hitchin fiber over $\underline{a}$
\begin{align*} T\Jac({\Sigma_{\underline{a}}})&=H^1({\Sigma_{\underline{a}}},\calO_{\Sigma_{\underline{a}}}) = H^1(X,\calO_X) \oplus H^1(X,\pi_*\calO_{\Sigma_{\underline{a}}}^-)
 = T\pi^*\Jac(X) \oplus T\Prym(\Sigma_{\underline{a}}).
\end{align*}
Let $\alpha \in H^0(X,K_X)$ and $\beta \in H^1({\Sigma_{\underline{a}}},\calO_{\Sigma_{\underline{a}}}) \cong H^{(0,1)}({\Sigma_{\underline{a}}})$, then
\[ \int_{\Sigma_{\underline{a}}} \pi^*\alpha \wedge \beta = \int_X \alpha \wedge \mathrm{nm}_{\Sigma/X} (\beta).
 \] Therefore, the Serre pairing on ${\Sigma_{\underline{a}}}$ restricts to a non-degenerate pairing between $H^0({\Sigma_{\underline{a}}},K_{\Sigma_{\underline{a}}})^-$ and $H^1({\Sigma_{\underline{a}}}, \calO_{{\Sigma_{\underline{a}}}})^-$. Consequently, the isomorphism $t$ of Proposition~\ref{prop:Baraglia} restricts to
\[ t\colon \ \bigoplus\limits_{i=2}^n H^0(X,K_X^i) \to H^0(\Sigma_{\underline{a}},K_{\Sigma_{\underline{a}}})^-
\] by the same formula.
Now Proposition~\ref{prop:criterion} readily translates to a characterization of visible Lagrangians in $\calM_{\SL(n,\CC)}$.

\subsection{Langlands duality for Hitchin systems}
In this section, we will review the Langlands duality of Hitchin system as considered in~\cite{DonagiPantev}. We will fix $G=\SL(n,\CC)$ and $G^{\rm L}=\PGL(n,\CC)$ (cf.\ Remark~\ref{rema:GL}). Recall that we consider the neutral components of the moduli spaces. The situation is visualized in Figure~\ref{fig:Langlands}. First there is an isomorphism of the Hitchin bases for $G$ and $G^{\rm L}$, which maps the $G$-discriminant locus to the $G^{\rm L}$-discriminant locus. In the cases under consideration, this isomorphism is the identity. The $\SL(n,\CC)$-Hitchin system restricted to \smash{$\calB_{\SL(n,\CC)}^\reg$} is a torsor over the abelian scheme \smash{$\Prym(\Sigma/\calB^{\reg}_{\SL(n,\CC)})$}. On the other hand, the (neutral component of the) moduli space of $\PSL(n,\CC)$-Higgs bundles is the quotient
\[ \calM_{\PSL(n,\CC)} = \calM_{\SL(n,\CC)}/\Jac(X)[n]
\] under the action of $\Jac(X)[n]$ on the moduli spaces of $\SL(n,\CC)$-Higgs bundles by tensor product. The $\PGL(n,\CC)$-Hitchin system restricted to \smash{$\calB^\reg_{\PSL(n,\CC)}$} is a torsor for the dual abelian scheme
\[
\Prym\bigl(\Sigma/\calB^{\reg}_{\SL(n,\CC)}\bigr)^\vee=\Prym\bigl(\Sigma/\calB^{\reg}_{\SL(n,\CC)}\bigr)/\pi^*\Jac(X)[n].
\] In the case of $\SL(n,\CC)$ and $\PGL(n,\CC)$, the quotient map $\delta\colon \calM_G \to \calM_{G^{\rm L}}=\calM_G/\Jac(X)[n]$ extends the polarization to a finite morphism between the moduli spaces. This morphism is holomorphic symplectic, i.e., it induces an symplectic isomorphism of tangent spaces at the points with non-trivial stabilizer.
\begin{figure}[t]
 \centering
 \begin{tikzcd}
 \calM_G \ar[rd]\ar[rr] & & \calM^0_{G^{\rm L}}
 \ar[ld] \\
 & \calB_G=\calB_{G^{\rm L}} \ar[lu,bend left=20,"s_H" below]\ar[ru,bend right=20,"s_H" below] & \\
\Prym\bigl(\Sigma/\calB^{\reg}_{G}\bigr) \ar[rd]\ar[uu,bend left=50] & & \Prym\bigl(\Sigma/\calB^{\reg}_{G}\bigr)/\pi^*\Jac(X)[n] \ar[ld] \ar[uu,bend right=40] \\
 & \calB_G^\reg = \calB_{G^{\rm L}}^\reg \ar[uu,hook] \ar[lu,bend left=15,"e" above]\ar[ru,bend right=15,"e" above] &
 \end{tikzcd}

 \vspace{-2mm}

 \caption{Langlands duality between $\SL(n,\CC)$-Hitchin system and $\PGL(n,\CC)$-Hitchin system.}\label{fig:Langlands}
\end{figure}
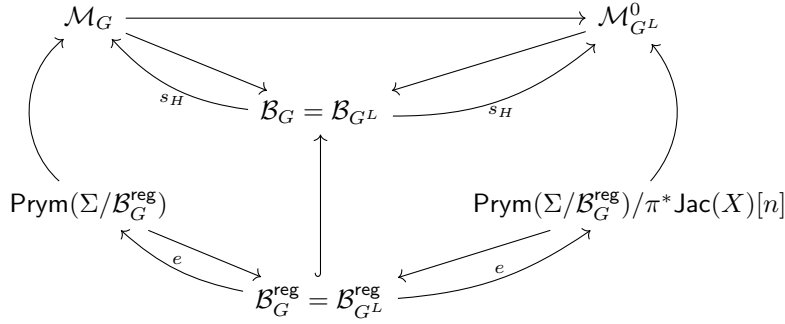

\begin{theo}\label{theo:subintegrable_syst} Let $\calL \to \calB'\subset \calB_{G}$ be a visible Lagrangian in $\calM_G$ with $\calB'^\reg \subset \calB^\reg$, connected fibers and a section $s\colon \calB'^\reg \to \calL\rest_{\calB'^\reg}$. Then there exists a holomorphic symplectic subvariety $\calI \subset \calM_{G^{\rm L}}$, such that $\calI \to \calB'\subset \calB_{G^{\rm L}}$ is an algebraically completely integrable system and $s'=\delta \circ s$ defines a section~$s'\colon \calB'^\reg \to \calI\rest_{\calB'^\reg}$.
\end{theo}

\begin{proof} Using the section $s$, we can identify the Hitchin system over $\calB'^\reg$ with the abelian scheme $\Prym(\Sigma/\calB'^\reg) \to \calB'^\reg$. By Theorem~\ref{theo:visible_Lagr}, the fiber $\calL_b$ for $b \in \calB'^\reg$ is a complex subtorus that by assumption contains $s(b)$. Hence, $\calL$ defines an abelian subscheme $A \subset \Prym(\Sigma/\calB'^\reg)$. We obtain an exact sequence of abelian schemes over $\calB'^\reg$
 \begin{align} 0 \to A \to \Prym(\Sigma/\calB'^\reg) \to Q \to 0, \label{eq:exact_sequ}
 \end{align}
 where the quotient $Q$ is again an abelian scheme over $\calB'^\reg$. Dually, we obtain an exact sequence of abelian schemes
 \begin{align} 0 \to Q^\vee \to \Prym(\Sigma/\calB'^\reg)^\vee \to A^\vee \to 0. \label{eq:dual_exact_sequ}
 \end{align}
 We can define a section \smash{$s'= \delta \circ s\colon \calB'^\reg \to \Hit^{-1}_{G^{\rm L}}(\calB'^\reg) \subset \calM_{G^{\rm L}}$}. Acting by $Q^\vee$ on $s'\colon \calB'^\reg \to \calM_{G^{\rm L}}$ defines a submanifold \smash{$I \subset \Hit^{-1}(\calB'^\reg) \subset \calM_{G^{\rm L}}$}. We define $\calI=\overline{I}$. We want to show that this is subintegrable system with regular locus $I$, i.e., the holomorphic symplectic form restricts to a non-degenerate form on~$I$ and $\Hit\rest_{\calI}: \calI \to \calB'^\reg$ is an integrable system. First note that for all $b \in \calB'^\reg$
 \[\dim \calI_b=\dim(\Prym(\Sigma_b))-\dim \calL_b=\frac 12 \dim \calM(X,G)-\dim \calL_b=\dim B'^\reg.
 \]
 Hence, the restricted Hitchin map will lead the correct number of commuting Hamiltonian functions and the fibers are complex tori by definition. To show that it forms a subintegrable system, it suffices to show that the tangent space of $\calI$ at the section~$s'$ is a~symplectic vector space with the restriction of the symplectic form. The argument at a~general point follows by a~translation along the fibers of~$\calI$. We identify the tangent spaces $T_{s(b)}\calM_G \cong T_{s'(b)}\calM_{G^{\rm L}}$ using~$\delta$. Then the vertical tangent spaces to $V\calL$ and $V\calI$ are complementary by definition. Denote by~$H'_b=Ds'(T_b\calB'^\reg)$ the tangent space to the image of the section~$s'$ at~$b$. Tangent vectors in~$H'_b$~pair to zero with the vertical tangent vectors in $V_b\calL$ as $\calL$ is Lagrangian. On~the other hand, the symplectic form on $T_{s'(b)}\calM_{G^{\rm L}}$ is non-degenerate. Hence, it induces an isomorphism $(H'_b)^\vee \cong V_b\calI$ or equivalently the symplectic form restricts to a non-degenerate symplectic form on $T_{s'(b)} \calI=H'_b \oplus V_b \calI$.
\end{proof}

\begin{theo}\label{theo:FM} Let $\calL \to \calB'\subset \calB_{G}$ be a visible Lagrangian in $\calM_G$ with $\calB'^\reg \subset \calB^\reg$, connected fibers and a section $s\colon \calB'^\reg \to \calL$. We identify the $G$- $($resp.\ $G^{\rm L}$-$)$ Hitchin system over $\calB'^\reg$ with the abelian schemes using the section~$s$ $($resp.\ $s'=\delta \circ s)$. Then the fiber-wise Fourier--Mukai transform of the structure sheaf $\calO_{\calL}$ over $\calB'^\reg$ is the structure sheaf of the holomorphic symplectic subvariety $\calI \subset \calM_{G^{\rm L}}$ defined in Theorem~{\rm \ref{theo:subintegrable_syst}}.
\end{theo}
\begin{proof}
 Let $\underline{a} \in \calB'^\reg$ and denote by $P=\Prym(\Sigma_{\underline{a}})$. As in the previous proof we use the section $s$ to obtain the exact sequences of abelian varieties
 \[ 0 \to A \to P \to Q \to 0 \quad (\ref{eq:exact_sequ}) \qquad \and \qquad 0 \to Q^\vee \to P^\vee \to A^\vee \to 0 \quad (\ref{eq:dual_exact_sequ}).
 \]
 We use the symmetric Fourier--Mukai transform introduced in~\cite{Schnell}. More precisely, we are going to use Proposition~1.6 therein. We denote by $\calP \to P \times P^\vee$ the Poincar\'e bundle. Then the symmetric Fourier--Mukai transform is defined by
 \[ \SFM_P\colon \ D^b(P) \to D^b\bigl(P^\vee\bigr), \qquad \FM_{\calP} \circ \Delta_P,
 \]
 where $\Delta_P=\Hom(\cdot,\omega_P[\dim P])$ is the Serre duality functor. First we have to show that $\calO_A$ is a GV-sheaf. By \cite[Definition~3.1]{PareschiPopa}, this means that the support loci
 \[ \bigl\{ \xi \in P^\vee \mid H^i\bigl(P,\iota_*\calO_A \otimes \calP_\xi^{-1}\bigr)\neq 0 \bigr\}
 \]
 have codimension $\geq i$ for all $i$. Let $\xi \in P^\vee$. We have
 \[ H^i\bigl(P,\iota_*\calO_A \otimes \calP_\xi^{-1}\bigr)=H^i\bigl(A,\calP_\xi^{-1}\rest_{A}\bigr).
 \]
 It is zero for $i > \dim A$. For $i \leq \dim A$, it is non-zero if and only if $\calP_\xi^{-1} \rest_A$ is trivial. Hence, if and only if $\xi \in Q^\vee$. Therefore, the support loci satisfy the dimension condition and $\iota_*\calO_A$ is a GV sheaf on $P$. In particular, it is WIT -- that is $\SFM_P(\calO_A)$ is a sheaf -- by \cite[Proposition~3.2]{PareschiPopa}.

 The symmetric Fourier--Mukai transform on $A$ has the property
 \begin{align} \label{eq:FM_of_point} \SFM_A(\calO_A)=\CC_0, \qquad \SFM_{A^\vee}(\CC_0) = \calO_{A},
 \end{align}
 where $\CC_0$ is the skyscraper sheaf of length $1$ at $0 \in P^\vee$. Now \cite[Proposition~1.6]{Schnell} states that if a sheaf $\calF$ is GV on $A$ and $\iota_*\calF$ is GV on $P$, then
 \[ \SFM_P(\iota_*\calF)=\bigl(\iota^\vee\bigr)^*\SFM_A(\calF).
 \]
 By~(\ref{eq:FM_of_point}), this yields $\SFM_P(\calO_A)=\calO_{Q^\vee}$.
\end{proof}

\begin{rema}\label{rema:GL} The above arguments similarly work for the case of $G=G^{\rm L}=\GL(n,\CC)$. Here the Hitchin system restricted to \smash{$\calB_{\GL(n,\CC)}^{\reg}$} is a torsor over the abelian scheme \smash{$\Jac(\Sigma) \to \calB_{\GL(n,\CC)}^{\reg}$}, which is self-dual due to the principal polarization of the Jacobians. Hence, again given a visible Lagrangian with connected fibers together with a section we obtain a holomorphic symplectic submanifold $\calI$ by the arguments of the proof of Theorem~\ref{theo:subintegrable_syst}. Furthermore, the arguments in the proof of Theorem~\ref{theo:FM} work for the structure sheaf of any abelian subvariety of an abelian variety.
\end{rema}

\subsection{Visible Lagrangians in Hitchin systems}\label{sec:visi_Lagra_Hitchin}
In this subsection, we will explain two examples of visible Lagrangians in Hitchin systems. These examples are independent of the choice of the Riemann surface. The first one is associated to the embedding of the complex Lie groups $\SL(n,\CC) \subset \GL(n,\CC)$. The second example appeared in the work of \cite{FGOP} and is a subvariety of the singular fibers of the Hitchin map. The remainder of the paper will deal with visible Lagrangians defined on special Riemann surfaces.

\textit{Visible Lagrangian associated to the subgroup $\SL(n,\CC) \subset \GL(n,\CC)$:}
Consider the embedding $\calB_{\SL(n,\CC)} \subset \calB_{\GL(n,\CC)}$. In this subsection, we will define a visible Lagrangian over $\calB'=\calB_{\SL(n,\CC)}$.

\begin{theo}\label{theo:Lagr_SL_in_GL} Let $\calB'=\calB_{\SL(n,\CC)}(X) \subset \calB_{\GL(n,\CC)}$ with \smash{$\calB'^\reg=\calB^{\reg}_{\SL(n,\CC)}$}. Then we can act by the trivial torsor $\Jac(X) \times \calB'$ on the Hitchin section by tensoring the underlying bundle. The orbit defines a visible Lagrangian $\calL \to \calB'$. The fiber-wise Fourier--Mukai transform of $\calO_{\calL}$ over $\calB'^\reg$ is supported on the moduli space of~$\SL(n,\CC)$-Higgs bundles $\calM_{\SL(n,\CC)} \subset \calM_{\GL(n,\CC)}$.
\end{theo}
\begin{proof} Tensoring with a line bundle preserves stability. Hence $\calL$ is well defined.
 The Hitchin fiber over $\underline{a} \in \calB'^\reg$ reflects the splitting of the differentials on $\Sigma$ by the isomorphism
\[ \Jac(\Sigma_{\underline{a}})= \pi^*\Jac(X) \times \Prym(\Sigma_{\underline{a}})/\pi^*\Jac(X)[n],
\] where $\pi^*\Jac(X)[n]$ acts diagonally (see Mumford \cite{Mumford}). In particular, we have an exact sequence of abelian schemes over $\calB'^\reg$
\begin{align} 0 \to \Jac(X) \times \calB'^\reg \xrightarrow{\iota} \Jac(\Sigma/\calB'^\reg) \to \Prym(\Sigma/\calB'^\reg)/ \pi^* \Jac(X)[n] \to 0. \label{eq:sequence_Jac_Prym}
\end{align}
 As explained in Section~\ref{sec:prel-Hitchin_syst}, the holomorphic symplectic form on \smash{$\Hit^{-1}\bigl(\calB^{\reg}_{\GL(n,\CC)}\bigr) \subset \calM_{\GL(n,\CC)}$} restricts to the Serre pairing between the tangent space of the base identified with $H^0({\Sigma_{\underline{a}}},K_{\Sigma_{\underline{a}}})$ by Proposition~\ref{prop:Baraglia} and the tangent space to the fibers $H^1({\Sigma_{\underline{a}}}, \calO_{\Sigma_{\underline{a}}})$. With respect to this pairing, we have
 \[ H^1(X,\calO_X)^\perp=H^0({\Sigma_{\underline{a}}},K_{\Sigma_{\underline{a}}})^-.
 \]
 The subtorsor $\calL \subset \Hit^{-1}(\calB'^\reg)$ that is defined by the action of $\Jac(X) \times \calB'^\reg$ on the Hitchin section has vertical tangent bundle $H^1(X,\calO_X)$ and the tangent space to the $\SL(n,\CC)$ Hitchin base are identified with $H^0(\Sigma_{\underline{a}},K_{\Sigma_{\underline{a}}})^-$. Therefore, $\calL$ is a visible Lagrangian.

 By Theorem~\ref{theo:FM}, the fiber-wise Fourier--Mukai transform of the structure sheaf of $\calL$ over $\calB'^\reg$ is supported on torsor over the abelian scheme that is dual to quotient in the exact sequence~(\ref{eq:sequence_Jac_Prym}). That is the abelian scheme \smash{$\Prym\bigl(\Sigma/ \calB^{\reg}_{\SL(n,\CC)}\bigr)$}. The closure of this locus is the moduli space of $\SL(n,\CC)$-Higgs bundles.
\end{proof}

We expect that every embedding of complex reductive Lie groups $G_1 \subset G_2$, such that center of $G_1$ is mapped to the center of $G_2$ defines a visible Lagrangian
\[ \calL \to \calB_{G_1^{\rm L}}\subset \calB_{G_2^{\rm L}}\cong \calB_{G_2}
 \] Fourier--Mukai dual to the image of the induced morphism of moduli spaces $\calM_{G_1} \rightarrow \calM_{G_2}$.

By definition, the visible Lagrangian of Theorem~\ref{theo:Lagr_SL_in_GL} has connected fibers. We provide an example with disconnected fibers by acting on a multi-section of the $\SL(2,\CC)$-Hitchin map instead of a section. Consider the moduli space of $\SL(2,\RR)$-Higgs bundles $\calM_{\SL(2,\RR)}$. The Hitchin map restricts to a $2^{6g-6}$-covering
\[ \Hit\rest\colon \ \calM_{\SL(2,\RR)} \cap \Hit^{-1}\calB_{\SL(2,\CC)}^\reg \to \calB_{\SL(2,\CC)}^\reg
 \]
(see, for example, \cite[Corollary~9.3]{Horn1}).
\begin{theo}\label{theo:Lagr_SL_in_GL_multisection} Let $\calB'=\calB_{\SL(2,\CC)} \subset \calB_{\GL(2,\CC)}$ with \smash{$\calB'^\reg=\calB^{\reg}_{\SL(2,\CC)}$}. We can act by the trivial torsor $\Jac(X) \times \calB'$ on $\calM_{\SL(2,\RR)} \to \calB'$ by tensoring. The orbit defines a visible Lagrangian $\calL \to \calB'$. For $\underline{a} \in \calB'^\reg$, the fiber-wise Fourier--Mukai transform of $\calO_{\calL_{\underline{a}}}$ is supported on $\calM_{\SL(2,\CC)} \subset \calM_{\GL(2,\CC)}$ and given by a flat vector bundle of rank $6g-6$ on $\Prym(\Sigma_{\underline{a}})$.
\end{theo}

\begin{proof}
 Let $U \subset \calB^{\reg}_{\SL(2,\CC)}$ open, such that there exist sections
 \[ s_1=s_H, s_2, \dots, s_{6g-6}\colon \ U \to \calM_{\SL(2,\CC)}
 \]
 such that $\Hit^{-1}(U) \cap \calM_{\SL(2,\RR)} = \bigsqcup_{i=1}^{6g-6} s_i(U)$. Then the orbit of each $s_i(U)$ under tensoring with line bundles in $\Jac(X)$ is Lagrangian by the previous proof. Hence, $\calL$ defines a visible Lagrangians.

 We identify the Hitchin fibers with the abelian scheme $\Jac(\Sigma)$ using the Hitchin section~$s_1$. Fix $\underline{a} \in U$ and define $l_i = (s_1-s_i)(\underline{a}) \in \Jac(\Sigma_{\underline{a}})$. Recall the exact sequence of abelian schemes~(\ref{eq:sequence_Jac_Prym}). We have
 \[ \calO_{\calL_{\underline{a}}} = \bigoplus\limits_{i=1}^{6g-6} t^*_{l_i} \iota_* \calO_{\Jac(X)},
 \] where $t_{l_i}$ denotes the translation by $l_i$ on $\Jac(\Sigma_{\underline{a}})$. Let $J=\Jac(\Sigma_{\underline{a}})$ and $\calP \to J \times J$ the Poincar\'e bundle. Using the relation between tensor product and translations under Fourier--Mukai transform (see \cite{Schnell}), we obtain
 \begin{align*} \SFM_{J}(t^*_{l_i} \iota_*\calO_{\Jac(X)}) = \calP_{l_i} \otimes \SFM_{J}(\iota_* \calO_{\Jac(X)})= \calP_{l_i} \otimes \bigl(\iota^\vee\bigr)^* \CC_0 = \calP_{l_i} \otimes \calO_{\Prym(\Sigma_{\underline{a}})}.
 \end{align*}
 The right hand side defines a flat line bundle on $\Prym(\Sigma_{\underline{a}})$. The Fourier--Mukai of the structure sheaf $\calO_{\calL_{\underline{a}}}$ is the direct sum of all these flat line bundles.
\end{proof}

\textit{Visible Lagrangians over the singular locus of the Hitchin base:}
Other examples of visible Lagrangians were considered in~\cite{FGOP}. Here $\calB'^\reg \subset \calB_{\GL(n,\CC)}$ is the locus of spectral curves with the maximal number of $n(n-1)(g-1)$ nodes as their only singularities and $\calB'=\overline{\calB'^\reg}$. In particular, these Lagrangians are completely contained in the singular locus of the Hitchin map and hence Theorems~\ref{theo:visible_Lagr} and~\ref{theo:FM} do not apply. The compactified Jacobians over $\calB'^\reg$ contain subvarieties isomorphic to \smash{$\bigl(\PP^1\bigr)^{n(n-1)(g-1)}$}, which can be interpreted as parameters for Hecke modifications of the Higgs bundles at the node by~\cite{Horn1}. Applying these Hecke modifications to the Hitchin section yields a visible Lagrangian over~$\calB'$. Interestingly, in this work the authors considered Arinkin's Fourier--Mukai transform for compactified Jacobians and found that the support of the fiber-wise Fourier--Mukai transform is supported on a hyperholomorphic subvariety -- the so-called Narasimhan--Ramanan BBB-brane. In the subsequent work~\cite{FHHO}, together with the first author this construction will be generalized to visible Lagrangians $\calL \to \calB'$ over the closure of the locus of spectral curves with any number of nodes as their only singularities.

\textit{Lagrangians that are not visible:} The upward flow of a very stable Higgs bundle considered in~\cite{HauselHitchin, PaulyPeon} defines a complex Lagrangian that is supported over the whole Hitchin base and hence is not visible.

\section{Parallelogram-tiled surfaces and pillowcase covers}\label{sec:rhomb-tiled-surf}

In this section, we will briefly review the interpretation of abelian and quadratic differentials in terms of flat geometry. Then we will discuss certain types of these flat geometries on a Riemann surface that will play a special role in the following section.

An abelian differential $\lambda \in H^0(\Sigma,K_\Sigma)$ on a Riemann surface $\Sigma$ determines a singular flat metric, such that all transition functions are translations. Denote by $Z(\lambda) \subset \Sigma$ the zeros of $\lambda$. A coordinate of the flat metric at $y \in \Sigma \setminus Z(\lambda)$ is a holomorphic coordinate $z$ at $y$, such that $\lambda= \d z$. In this way, one obtains a flat metric on $\Sigma\setminus Z(\lambda)$, such that coordinate transitions are translation. It extends to a singular flat metric on $\Sigma$ by cone points of cone angle $(k+1)\pi$ at a~zero of $\lambda$ of order $k$. This is a so-called translation surface.

Similarly, we can associate a singular flat metric to a quadratic differential $(X,q)$, where $X$ is a Riemann surfaces and \smash{$q \in H^0\bigl(X,K_X^2\bigr)$}. A flat coordinate at $x \in X\setminus Z(q)$ is a holomorphic coordinate $z$ at $x$, such that $q= \d z ^{\otimes 2}$. In this case, the coordinate functions are compositions of translations and reflections. It extends to a singular flat metric on $X$ by cone points of cone angle $(k+2)\pi$ at a zero of $q$ of order $k$. This is a so-called half-translation surface. When $q$ has simple zeros only, the spectral curve defined in Section~\ref{sec:hitchin-systems} is referred to as the canonical cover of~$(X,q)$ from this point of view. It is the unique branched double cover of $X$, such that the pullback of~$q$ has a square-root. (Here we consider $\lambda$ as a~section of $K_\Sigma$ instead of $\pi^*K_X$ as in Section~\ref{sec:hitchin-systems}. We have $\pi^*K_X=K_\Sigma(-R)$. Hence, if $q$ has simple zeros, then the abelian differential $\lambda$ has double zeros at all branch points.)

We say that a quadratic differential $(X,q)$ is of type $\mu(q) = (m_1,\dots,m_n)$ if the orders of the zeros of the differential are $m_1,\dots,m_n$.
We will use exponential notation if multiple $m_i$ agree, i.e., we write \smash{$\bigl(1^{4g(X)-4}\bigr)$} for $(1, \dots, 1)$.

In the following, particularly symmetric (half-)translation surfaces play a special role: Para\-llelo\-gram-tiled surfaces and pillowcase covers. We obtain coordinates on the stratum of abelian differentials with fixed distribution of zeros by recording periods. The periods of $(\Sigma,\lambda)$ in $H^1(\Sigma,Z(\lambda),\CC)$ are given by
\[ H_1(\Sigma, Z(\lambda),\CC) \to \CC, \qquad c \mapsto \int_c \lambda
\]
and local coordinates, the so called \emph{period coordinates}, are given by the image of a basis of relative homology under this map.
The coordinate changes of period coordinates are induced by diffeomorphisms of $(\Sigma,Z(\lambda))$ and hence preserve the lattice
\[ H^1(\Sigma, Z(\lambda), \ZZ \oplus i \ZZ) \subset H^1(\Sigma,Z(\lambda),\CC).
\]

These integral points correspond to square-tiled surfaces: One obtains a cover of an elliptic curve by
\[ p\colon \ \Sigma \to \CC /\ZZ \oplus i\ZZ, \qquad y \mapsto \int_{y_0}^y \lambda
\]
for a choice of base point $y_0 \in \Sigma$. This cover is branched over one point $0 \in E$ with ramification points $z_i \in Z(\lambda)$ and ramification profile $(\ord_{z_1}(\lambda),\dots,\ord_{z_n}(\lambda))$. In particular, this is a cover of flat surfaces, i.e., $\lambda=p^*\omega$ for some abelian differential $\omega$ on $E$. By rescaling the flat torus $(E,\omega)$, we obtain a dense subset of square-tiled surfaces in each stratum. We use a slight generalization.

\begin{defi} A translation surface $(\Sigma, \lambda)$ is parallelogram-tiled if and only if there exists a branched cover $p\colon \Sigma \to E$ branched over one point, such that $\lambda=p^*\omega$ for some abelian differential $\omega$ on $E$.
\end{defi}

Given a parallelogram-tiled surface, we can act by $\GL^+(2,\RR)$ on the representing polygon and obtain a family of parallelogram-tiled surfaces over the $j$-line, the moduli space of elliptic curves. For $j=1728$, we recover a square-tiled surface. The analogue of parallelogram-tiled surfaces for quadratic differentials are pillowcase covers. A~pillowcase is a half-translation surface $\bigl(\PP^1,\eta\bigr)$, where~$\eta$ has four simple poles $0$, $1$, $\infty$, $x$, see Figure~\ref{fig:Pillowcase}.

\begin{figure}[t]
 \centering
	\begin{tikzpicture}
		\begin{scope}[thick,decoration={
 			markings,
 			mark=at position 0.5 with {\arrow{>}}}
 			]
 			\draw[postaction={decorate}] (0,0.5)--(1,0.5);
 			\draw[postaction={decorate}] (2,0.5)--(1,0.5);
 			\draw[postaction={decorate}] (0,1.5)--(1,1.5);
 			\draw[postaction={decorate}] (2,1.5)--(1,1.5);
 			\draw (0,0.5)--(0,1.5)--(1,1.5)--(1,0.5)--(2,0.5)--(2,1.5);
 				\foreach \x in {0,1,2}
 					\foreach \y in {0.5,1.5}
 						{ \draw [fill=red,red] (\x,\y) circle [radius=0.04];
 						}
 						
 			\draw[very thick,<-] (3,1)--(4,1);

 			\draw (5,0)--(6,0)--(7,0)--(7,1)--(7,2)--(6,2)--(5,2)--(5,1)--(5,0);
 			\draw (6,0)--(6,1)--(6,2);
 			\draw (5,1)--(6,1)--(7,1);
 			\foreach \x in {5,6,7}
 				\foreach \y in {0,1,2}
 					{\draw [fill=red,red] (\x,\y) circle [radius=0.04];
 					}
 			\draw [<->,blue] (6.2,0.8)--(5.8,1.2);
 		\draw [<->,blue] (5.8,0.8)--(6.2,1.2);
		\end{scope}
	\end{tikzpicture}\vspace{-2mm}

	\caption{Pillowcase with canonical cover: Opposite sides identified, when not indicated otherwise. The involution on the cover acts as central symmetry in the two-torsion points.} \label{fig:Pillowcase}
\end{figure}
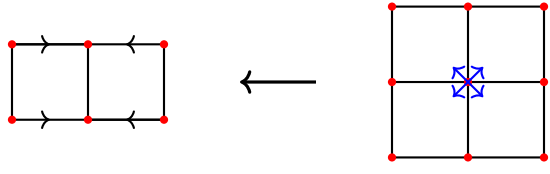

\begin{defi} A half-translation surface $(X,q)$ is called pillowcase cover if there exists a cover $ \check{p}\colon X \to \PP^1$ branched over $D=(0,1,\infty,x)$, such that $q= \check{p}^*\eta$ for a quadratic differential $\eta$ on $\PP^1$ with simple poles at~$D$.
\end{defi}
The canonical cover of $\bigl(\PP^1,\eta\bigr)$ is the elliptic differential $(E,\omega)$, where $E=\CC/(\ZZ+\tau\ZZ)$ with $\lambda(\tau)=x$ and the involution is given by multiplication with $-1$. Here $\lambda$ is the modular lambda function.
For a given $x$, the number $\tau$ can be computed explicitly.
Let $K$ denote the complete elliptic integral of the first kind
\begin{equation} \label{eq:elliptic_integral}
	K(k) = \int_0^{\frac{\pi}{2}} \frac{\mathrm{d}\theta}{\sqrt{1-k^2\sin^2\theta}}.
\end{equation}
Then $t(x) = i\frac{K(\sqrt{1-x})}{K(\sqrt{x})}$ is a section of $\lambda$.
There is the following well-known relation between para\-llelo\-gram-tiled surfaces and pillowcase covers.

\begin{lemm}\label{lemm:pillow} Let $(X,q)$ be a half-trans\-lation surface. The canonical cover $(\Sigma,\lambda)$ is a para\-llelo\-gram-tiled surface if and only if the quadratic differential is a pillowcase cover.
\end{lemm}
\begin{proof}
 Starting from a pillowcase cover as above, we have the following diagram:
\[	\begin{tikzcd}
		& (\Sigma,\lambda) \ar[rd,"p",dashed] \ar[ld,"\pi" above, "\mod J" right]& \\ (X,q) \ar[rd,"\check{p}"] & & (E, \omega) \ar[ld,"2:1"] \\ & \bigl(\PP^1,\eta\bigr) &
\end{tikzcd}
\]
and want to show that there exist the dashed arrow $p\colon \Sigma \to E$, such that the diagram commutes. Away from the singularities of the flat structure associated with $(X,q)$ a point on $\Sigma$ corresponds to a choice of a (local) square of~$q$. This corresponds to a choice of a local square root of $\eta$ under the covering map to $\PP^1$. Hence, there is a induced map $p\colon \Sigma\setminus Z(\lambda) \to E \setminus Z(\omega)$ such that $p^*\omega=\lambda$.
The map $p$ uniquely extends to a map of Riemann surfaces $\Sigma \to E$. A zero $x \in X$ of $q$ of order $k\geq 0$ corresponds to a $(k+2):1$-branch point of $\check{p}$. Hence, the map can be extended to $p\colon \Sigma \to E$ by gluing in a $k:1$ ramification point, if $k$ is odd and a $\frac{k}{2}:1$ ramification point, if~$k$ is even, at $y \in \pi^{-1}(x)$.
 By construction the relation of the differentials persists under this extension. Taking the quotient of $E$ by two-torsion points we obtain a map $(\Sigma,\lambda) \to E/E[2]$, branched over one point such that $\lambda$ is the pullback of a differential on~$E$. Hence, it is a parallelogram-tiled surface.

For the converse, we start with the configuration
\[ 	\begin{tikzcd}
		& (\Sigma,\lambda) \ar[rd,"p"] \ar[ld,"\pi" above left, "\mod J" right]& \\ (X,q) & & (E, \omega).
	\end{tikzcd}
\]
The differential $\lambda$ is anti-symmetric with respect to an involution~$J$. In particular, $J$ sends singularities of the flat structure to singularities of the same type, saddle connections to saddle connections and hence squares to squares. Hence, it descends to the elliptic curve $(E,\omega)$ to the reflection in the two torsion points of~$E$. The quotient is the pillowcase surface as illustrated in Figure~\ref{fig:Pillowcase}. In particular, there is an induced map $\check{p}\colon \Sigma/J=X \to \PP^1$, such that $q=\check{p}^*\eta$.
\end{proof}

\begin{rema}
 The proof of the previous lemma shows that one can always assume the pillowcase cover to be non-simply branched over only one point on $\PP^1$. This can be achieved by taking the quotient by the order four automorphism of $\PP^1$ that permutes the four marked points preserving the cross-ratio.
\end{rema}

In the following, a special kind of pillowcase cover will be important.
\begin{defi}\label{defi:uniform_pillow_cover}
	We call a pillowcase cover $(X,q)$ \emph{uniform} if every fiber of $\check{p}\colon X \to \PP^1$ over $y \in D$ consists of ramification points of the same ramification index~$i$.
\end{defi}

\section{Multifold pillowcase covers} \label{sec:pill-sever-ways}

Motivated by the connection to visible Lagrangians in Theorem~\ref{theo:visi_Lagr_over_line}, we are interested in Riemann surfaces $X$ which admit multiple quadratic differentials $q_1,\dots,q_n \in H^0\bigl(X,K^{2}\bigr)$ such that
\begin{itemize}\itemsep=0pt
	\item the vanishing loci $Z(q_1), \dots, Z(q_n)$ are pairwise different,
	\item the half-translation surfaces $(X, q_i)$ are pillowcase covers.
 \end{itemize}
In Corollary \ref{coro:visible_Lagr_over_line_sl2C}, we will prove that those $X$ will allow for multiple visible Lagrangian subvarieties $\calL_i \subset \calM_{\SL(2,\CC)}$ projecting to the lines $\CC q_i \subset H^0\bigl(X,K^{2}\bigr)$.
\begin{defi}
	We call a Riemann surface $X$ with quadratic differentials $q_1,\dots,q_n$ as above a~\emph{$n$-fold pillowcase cover}.
	We call a $n$-fold pillowcase cover a \emph{multifold pillowcase cover} if $n \geq 2$.
\end{defi}

Apart from being uniform, we want the pillowcase covers to have simple zeros only.

Two quadratic differentials $q_1$, $q_2$ on $X$ are called isomorphic if there exists an automorphism~$\varphi$ of~$X$ such that $\varphi^* q_1 = q_2$.
We remark that from the point of view of flat geometry isomorphic differentials are usually not distinguished.
However two isomorphic differentials might still have different vanishing loci and therefore correspond to different points in the $\SL(2,\CC)$-Hitchin base. Hence, we will treat these differentials as different from each other in the following.
In this section we will prove the following.

\begin{theo}\label{theo:multifold_pillow}
	For infinitely many genera $g$, there exist multifold uniform pillowcase covers with simple zeros only.
\end{theo}
\begin{proof}
	Assume that we know a single multifold uniform pillowcase cover $X$ with simple zeros only of some genus $g \geq 2$.
	Let $q_1, \dots, q_n$ be the corresponding quadratic differentials.
	Then we can obtain examples in infinitely many genera by taking unramified coverings $f\colon \widehat X \to X$ in different degrees and the differentials $\widehat q_i = f^* q_i$.

We will provide explicit examples in the following.
\end{proof}

The following examples have been found by a computer search using~\cite{GAP4}. There are many more examples to be found, but we will restrict our discussion to three examples in low genus.
While the claimed properties of the examples can be checked by hand, it is much more convenient to use a computer algebra system.

\begin{exam}[genus~2]
Consider the group $\GL(2,F_3)$ of order $48$ and let $f\colon X\to \PP^1$ be the $\GL(2, F_3)$-cover branched above three points with monodromy datum
\[
	\begin{pmatrix}0&1\\1&2\end{pmatrix}, \qquad
	\begin{pmatrix}2&0\\0&1\end{pmatrix}, \qquad
	\begin{pmatrix}2&2\\1&0\end{pmatrix}.
\]
The orders of the matrices are $8$, $2$ and $6$, and the genus of $X$ is $2$.
The group $\GL(2,F_3)$ has $12$ subgroups of order $6$ which we denote by $H_1, \dots, H_{12}$.
For each such subgroup the quotient~$X/H_i$ is of genus $0$, and the quotient map $g_i\colon X \to X/H_i$ is branched above four points with ramification orders~$2$ and~$3$.
Let~$\eta_i$ denote the quadratic differential on $X/H_i$ of type~$(-1^4)$ whose simple poles are supported at the branch points of~$g_i$.
The pullback $q_i := g_i^*\eta_i$ is a~uniform pillowcase cover with simple zeros only on $X$.
We claim that the vanishing loci of the differentials~$q_i$ are pairwise different. This is a very explicit but lengthy computation which is left to the reader.
In particular, $X$~is a $12$-fold uniform pillowcase cover with simple zeros only. 	
\end{exam}

\begin{exam}[genus~3]
Consider the group $\SL(3,F_2)$ of order $168$ and let $f\colon X\to\PP^1$ be the $\SL(3,F_2)$-cover branched above three points with monodromy datum
\[
	\begin{pmatrix}1&0&0\\0&1&0\\1&0&1\end{pmatrix}, \qquad
	\begin{pmatrix}0&1&0\\0&0&1\\1&0&0\end{pmatrix}, \qquad
	\begin{pmatrix}1&0&1\\1&0&0\\0&1&0\end{pmatrix}.
\]
The orders of the matrices are $2$, $3$ and $7$. The Riemann surface $X$ is the Klein quartic and has genus $3$.
The group $\SL(3,F_2)$ has $14$ subgroups of order~$24$, and as in the previous example each of those subgroups gives rise to a uniform pillowcase covers with simple zeros only.
In particular, $X$ is a $14$-fold uniform pillowcase cover with simple zeros only.
\end{exam}

\begin{exam}[non-isomorphic differentials]\label{exam:bifold_nonisom}
Let $G := A_4 \times \ZZ/3\ZZ$.
We choose generators $\langle a:= (1\, 2\, 3), b:=(1\, 2)(3\, 4) \rangle = A_4$ and $\langle c \rangle = \ZZ/3\ZZ$.
Consider the subgroups $H_1 := \langle a, c \rangle \cong (\ZZ/3\ZZ)^2$
and $H_2 := \langle a, b \rangle \cong A_4$.
Let $f\colon X \to \PP^1$ be the $G$-cover branched above three points with monodromy datum $\bigl(bc, a^2c^2, ab\bigr)$.
The genus of~$X$ is~$4$.

Consider the two intermediate covers $A_1 := X/H_1$ and $A_2 := X/H_2$, both of genus $0$.
We~define the two differentials $q_1$ and $q_2$ on $X$ as in the previous examples.

In this example, it is relatively easy to see that the differentials $q_1$ and $q_2$ are non-isomorphic.
For this, it is convenient to consider the respective canonical covers $(\Sigma_i, \lambda_i)$ of~$(X, q_i)$.
\end{exam}
\begin{prop}\label{prop:non-isomorphic_pillow}
	The canonical covers $\Sigma_1$ and $\Sigma_2$ are non-isomorphic.
	In particular, $(X,q_1) \not \cong (X,q_2)$, i.e., there does not exist an automorphism $\varphi\colon X \to X$, such that $\varphi^* q_1 = q_2$.
\end{prop}
\begin{proof}
	The curves $\Sigma_1$ and $\Sigma_2$ are covers of $\PP^1$ via $\Sigma_i \to X \xrightarrow{f} \PP^1$, branched over three points.
	Both covers $\Sigma_i \to \PP^1$ have a monodromy representation with elements in the symmetric group $S_{2|G|}$.
	The covers $\Sigma_1$ and $\Sigma_2$ are isomorphic if and only if the elements of the monodromy representation are conjugated in $S_{2|G|}$, which is easily checked not to be the case.
\end{proof}

In the rest of this section, we will produce explicit flat pictures of the two quadratic differentials $(X, q_1)$ and $(X, q_2)$ of Example~\ref{exam:bifold_nonisom}.
Recall that $\Sigma_i$ is a square tiled surface by Lemma~\ref{lemm:pillow}.
We thus have the diagram
\[
\begin{tikzcd}[row sep=small]
	\Sigma_1\ar[ddrr, "2:1"]\ar[ddd] &&&& \Sigma_2\ar[ddll, "2:1"']\ar[ddd] \\
	\\
	&& X\ar[ddd]\ar[ddl]\ar[ddr] && \\
	E_1\ar[dr, "2:1"'] &&&& E_2\ar[dl, "2:1"] \\
	& A_1\ar[dr, "4:1"'] && A_2\ar[dl, "3:1"] & \\
	&& \PP^1 &&
\end{tikzcd}
\]

First we determine the tori $E_i$, which also determine the pillowcases $A_i$. We start with the torus $E_2$.
The map $A_2 \to \PP^1$ is cyclic of degree $3$ and totally ramified over two points, and unramified otherwise. One of the branch points of $E_2 \to A_2$ agrees with a ramification point of $A_2 \to \PP^1$, while the three other branch points of $E_2 \to A_2$ lie in one fiber of $A_2 \to \PP^1$.
Hence, we~may assume that the branch points of $E_2 \to A_2$ are $0$, $1$, $\zeta_3$ and $\zeta_3^2$.
Those four points have the cross-ratio $D\bigl(0,1;\zeta_3,\zeta_3^2\bigr) = \zeta_6^5$, and $\lambda(\zeta_3) = \zeta_6^5$, where $\lambda$ is the modular lambda function.
Hence up to isomorphism
\[
	E_2 \cong \CC / (\ZZ \oplus \zeta_3 \Z) \cong \CC / (\ZZ \oplus \zeta_6 \ZZ).
\]

The cross-ratio of the four branch points $0, 1, x,\infty\in\PP^1$ of $E_1 \to A_1$ is again uniquely determined by the ramification profile of the maps $E_1 \to A_1 \to \PP^1$ and is given by
\[
	D(0,\infty;1,x) = 15\sqrt{3}-26.
\]
One numerically computes $\tau := t\bigl(15\sqrt{3}-26\bigr)\approx 1 + 2.143182698915i$, where $t$ is the function from~\eqref{eq:elliptic_integral}, and for this $\tau$ we have
\[
	E_1 \cong \CC / (\ZZ \oplus \tau\ZZ).
\]

To obtain pictures of the pillowcase covers $X \to A_i$, we need to determine how to glue the copies of the pillowcase~$A_i$.
We describe how to obtain this information for a general pillowcase cover which is a~$G$-cover.
The idea is to compare two $G$-actions on the $|G|$-many copies of $A_i$.

Given a G-cover $X \to \PP^1$ of degree $d=|G|$ with monodromy datum $(g_1, \dots, g_4)$, the bijective map
\[
	m_g\colon \ G \to G, \qquad h \mapsto gh
\]
induces a map $\rho\colon G \to S_d$ when we identify $\sigma\colon G \to \{1, \dots, d\}$.
If $p \in \PP^1$ is a point (which is not a branch point), and we identify the fiber of $X \to \PP^1$ above $p$ with $\{1, \dots, d\}$ via $\sigma$, then the lift of a simple loop with basepoint $p$ around the $i$-th ramification point starting in $q \subseteq \{1, \dots, d\}$ will end in a point $\rho(g_i)(q)$.

\begin{figure}[!ht]
 \centering
	\includegraphics{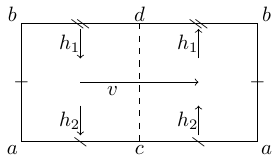}

\vspace{-2mm}

	\caption{A pillowcase.}	\label{fig:pillowcase}
\end{figure}

On the other hand, $X$ consists of $d$ copies of the pillowcase $\PP^1$ as depicted in Figure~\ref{fig:pillowcase}.
We~can again label those copies with $\{1, \dots, d\}$.
After choosing an orientation on the vertical and horizontal cylinder, $X$~is uniquely determined by the permutations $h_1, h_2, v \in S_d$, which indicate which copy of the pillowcase we reach when we leave a given copy in the direction indicated in the figure.

We label the four branch points by $a$, $b$, $c$, $d$ as in Figure~\ref{fig:pillowcase}.
Lifting a small clockwise cycle around point $a$ will act on the copies as $h_2 \circ v$, around $b$ as $h_1 \circ v^{-1}$, around $c$ as $h_2^{-1}$ and around~$d$ as~$h_1^{-1}$.

Hence, $h_1$, $h_2$, $v$ are given (up to the choice of $\sigma$ and hence $\rho$) by
\[
	h_1 = \rho(g_0)^{-1}, \qquad h_2 = \rho(g_1)^{-1}, \qquad v = h_2^{-1} \circ \rho(g_3) = \rho(g_4)^{-1} \circ h_1.
\]

Now we come back to the cases we are interested in.
For $X \to A_1$ we can choose $\sigma$ such that
\begin{align*}
	h_1 = h_2 = (1\;2\;3)(4\;5\;6)(7\;8\;9),\qquad
	v = (1\;5\;9)(2\;6\;7)(3\;4\;8),
\end{align*}
for $X \to A_2$ we can choose $\sigma$ such that
\begin{align*}
	h_1 = (1\;4)(2\;3)(5\;7)(6\;8)(9\;10)(11\;12),\qquad
	h_2 = v = (1\;7\;12)(2\;8\;11)(3\;5\;10)(4\;6\;9).
\end{align*}

Combined with the information about the tori $E_i$, this gives rise to the pictures in Figure~\ref{fig:covers}, where the horizontal edges are glued by half-translation as indicated by the labeling.

\begin{figure}[!ht]
		\centering
		\includegraphics{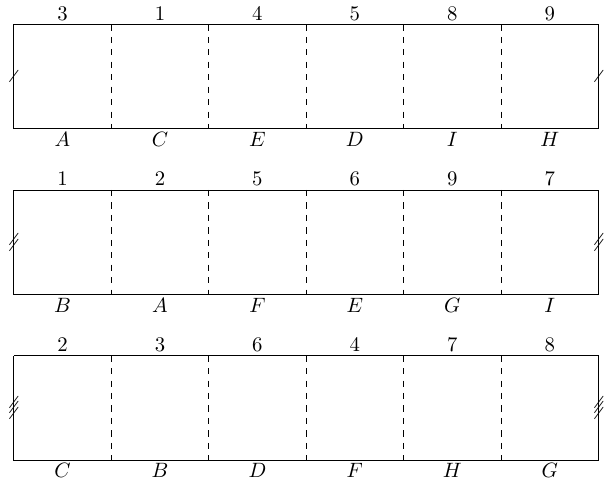}

{(a) The pillowcase cover $X \to A_1$}

\vspace{5mm}

		\includegraphics{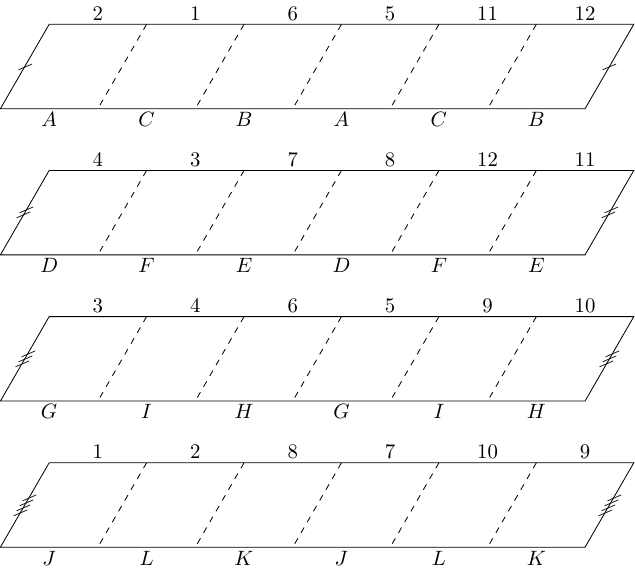}

{(b) The pillowcase cover $X \to A_2$}

	\caption{Two non-isomorphic pillowcase covers on the same curve $X$.}
	\label{fig:covers}
\end{figure}

\section{Visible Lagrangians over lines in the Hitchin base}

In this section, we study visible Lagrangians over a line $\calB'=\CC \underline{a} \subset \calB_{\SL(n,\CC)}$. In the first part, we~give an existence criterion using Proposition~\ref{prop:criterion}. In second part, for~$G=\SL(2,\CC)$ and $\underline{a}=q$ a~quadratic differential, we study the holomorphic symplectic subvariety~$\calI_q$, which is the proposed mirror dual by Theorem~\ref{theo:FM}. If $(X,q)$ is a uniform pillowcase cover, we will show that~$\calI_q$ is a~hyperholomorphic subvariety birational to Hausel's toy model.

\begin{theo}\label{theo:visi_Lagr_over_line} Let \smash{$\underline{a}=(0,\dots,0,a_n) \in \calB_{\SL(n,\CC)}^\reg$}. Then there exists a visible Lagrangian over~$\CC \underline{a}$, if and only if the spectral curve with its abelian differential $(\Sigma,\lambda)$ is parallelogram-tiled.
\end{theo}
\begin{proof}
 First, assume there exists a visible Lagrangian over~$\CC\underline{a}$. Let $(\Sigma,\lambda)$ be the spectral curve to $\underline{a}$. Then, by Proposition~\ref{prop:criterion}, there is an exact sequence of abelian varieties
 \[ 0 \to A \to \Prym(\Sigma_{\underline{a}}) \xrightarrow{\psi} E \to 0.
 \]
 Here $A$ is of codimension $1$ and $E$ is an elliptic curve. Under the assumption on $\underline{a}$, the spectral curve is $\ZZ_n$-Galois and $s_B=\lambda^{n-1}$ (up to a constant). Hence, the map $t$ of Proposition~\ref{prop:Baraglia} becomes
 \[ t\colon \ \bigoplus\limits_{i=2}^n H^0(X,K_X^i) \to H^0(\Sigma,K_\Sigma), \qquad \sum\limits_{i=2}^n \alpha_i X_i \mapsto \sum\limits_{i=2}^n \alpha_i \frac{\pi^*X_i}{\lambda^{i-1}}.
 \] In particular, the tangent vector $X_n=a_n$ is mapped to $\frac{\pi^*a_n}{\lambda^{n-1}}=\lambda$.
 Hence, the differential of the map $\psi$ can explicitly be written as
 \begin{align} \d \psi\colon \ H^1(\Sigma, \calO_\Sigma)^- \to T_0E, \qquad \alpha \mapsto c \int_\Sigma \alpha \wedge \lambda. \label{one-form}
 \end{align}
 Therefore, $\psi$ is given by $D=\sum a_i y_i \mapsto \sum a_i\int_{y_0}^{y_i} \lambda$ up to the choice of a point $y_0 \in \Sigma$. Denote by $\sigma$ a generator of the $\ZZ_n$ action on $\Sigma$. We want show that the composition $\psi \circ \mathrm{AP}$ with the Abel--Prym map
 \[ \mathrm{AP}\colon \ \Sigma \to \Prym(\Sigma), \qquad y \mapsto \calO(y-\sigma y)
 \]
 is a branched covering $p\colon \Sigma \to E$. Clearly, $\mathrm{AP}$ identifies all ramification points of $\pi\colon \Sigma \to X$. If~$\Sigma$ is not hyperelliptic, it is easy to see that these are the only points on $\Sigma$ that are identified. The spectral curve $\Sigma$ is never hyperelliptic by \cite[Lemma~4.1]{Baraglia}. By definition, $\sigma^*\lambda=\xi_n \lambda$ for some primitive $n$-th root of unity $\xi_n$. This implies that the differential of the composition $\psi \circ \mathrm{AP}$ is $(1-\xi) \lambda$. Hence, the differential of the composition is injective away from the ramification points of $\Sigma \to X$. In particular, $p=\psi \circ \mathrm{AP}\colon \Sigma \to E$ is a proper holomorphic map ramified at $Z(\lambda)$ over $0 \in E$. Furthermore, $\lambda$ considered as abelian differential has an order $n$ zero at each branch point and hence the points in $Z(\lambda)$ are $n+1:1$ ramifications. In particular, the pullback of an abelian differential~$\omega$ on~$E$ has the same divisor as~$\lambda$ and we can find a~specific~$\omega$ with $p^*\omega=\lambda$. Therefore, $(\Sigma,\lambda)$ is parallelogram-tiled.

 For the converse, let \smash{$\underline{a}=(0,\dots,0,a_n) \in \calB_{\SL(n,\CC)}^\reg$}, such that $(\Sigma,\lambda)$ is paralle\-lo\-gram-tiled, i.e., there is a covering $p\colon \Sigma \to E$, such that $\lambda=p^*\omega$ for some abelian differential $\omega$ on~$E$. The covering~$p$ induces a Norm map
 \[ \Nm_E\colon \ \Jac(\Sigma) \to E, \qquad D=\sum a_iy_i \mapsto \sum a_i \int_{y_0}^{y_i} \lambda.
 \] The restriction to $\Prym(\Sigma)$ defines the desired map
 \[ \psi=\Nm_E \rest_{\Prym}\colon \ \Prym(\Sigma) \to E.
 \] Let $D=(n-1)y - \sigma y - \dots - \sigma^{n-1}y$. Then $D \in \Prym(\Sigma)$ and $\Nm_E(D)= n \Nm_E(y)$. Hence, $\psi$~surjects onto $E$ being a divisible group. Its differential is the map~\eqref{one-form}. As $\Sigma$ does not change, when multiplying $q$ with a scalar, we can define an abelian subscheme $\ker(\psi) \subset \Prym(\Sigma/\CC^\times q)$ that satisfies the criterion of Proposition~\ref{prop:criterion}. To obtain a visible Lagrangian over~$\CC q$, we can act with the abelian scheme $\ker(\psi)$ on any section of $\Hit \rest_{\CC \underline{a}}$. To obtain a concrete example, we~may choose the Hitchin section.
\end{proof}

\begin{coro}\label{coro:visible_Lagr_over_line_sl2C}
 Let \smash{$q \in \calB_{\SL(2,\CC)}^\reg$} be a quadratic differential with simple zeros only. Then there exists a visible Lagrangian over~$\CC q$ if and only if~$(X,q)$ is a pillowcase cover.
\end{coro}
\begin{proof} This is immediate from Theorem~\ref{theo:visi_Lagr_over_line} and Lemma~\ref{lemm:pillow}.
\end{proof}

\subsection{The Fourier--Mukai dual}
Now, we consider the Fourier--Mukai dual of the visible Lagrangian defined above in the case of $G=\SL(2,\CC)$. Let $(X,q)$ be a pillowcase cover and $\calL_q \to \calB'=\CC q$ the visible Lagrangian defined as the closure of the orbit of the abelian subscheme $A \subset \Prym(\Sigma/\CC^\times q)$ on the Hitchin section. In particular, $\calB'^\reg=\CC^\times q$. The fibers of $\calL \to \calB'^\reg$ are of codimension 1 in the $\SL(2,\CC)$-Hitchin fibers over $\calB'^\reg$. Hence, by Theorem~\ref{theo:FM} the fiber-wise Fourier--Mukai transformation of the structure sheaf of $\calL$ is supported on an elliptic surface $I_q \rightarrow \CC^\times q$ obtained by acting with the abelian scheme $E \subset \Prym(\Sigma/\CC^\times q)^\vee$ on the Hitchin section of $\calM_{\PGL(2,\CC)}$. Its closure $\calI_q=\overline{I_q}$ is the proposed mirror dual. We have the following proposition.
\begin{prop}[\cite{HauselToy}]\label{prop:Hausel_toy} The subvariety $\calI_q \!\subset\! \calM_{\PSL(n,\CC)}$ is birational to Hausel's toy model~$\calM_\toy$.
\end{prop}
\begin{proof}
 Hausel's toy model is constructed as an elliptic surface as follows. Take $\bigl(\PP^1,p_1,p_2,p_3,p_4\bigr)$ and consider the elliptic curve $E \to \PP^1$, which is the canonical cover of the pillowcase with involution~$\tau$. Let further $\tau'$ be the involution $-1\colon \CC \to \CC$. Consider $M=E \times \CC/(\tau \times \tau')$. This orbifold has 4 $\ZZ_2$-points $\wh{p}_i \times 0$. The projection to the second factor induces the map $M \to \CC$, $(x,y) \mapsto y^2 $ with generic fiber $E$. Blowing up the 4~orbifold points, we obtain a smooth surface~-- the toy model $\calM_{\toy}$.

 The subintegrable system $\calI_q$ associated to $\calB'=\CC q$ for a pillowcase cover $(X,q)$ is an elliptic fibration over $\CC^\times q \subset \calB'$. The modulus of the elliptic curve is constant and determined by the four points $p_1,\dots,p_4$ on the pillowcase. As for Hausel's toy model, the monodromy around $0 \in \CC q$ is given by $-1$. In fact, if we scale our quadratic differential with $e^{i\phi}$, the abelian differential is multiplied by $e^{i\frac12 \phi}$. This defines an isomorphism between $M\setminus{\{y=0\}}$ and $\calI_q \rest_{\CC^\times q}$.
\end{proof}

For uniform pillowcase covers $(X,q)$, we can indeed prove that $\calI_q \subset \calM_{\SL(2,\CC)}(X)$ is a hyperholomorphic subvariety. This will be achieved by defining a morphism $\Theta$ from a moduli space of semi-stable parabolic $\SL(2,\CC)$-Higgs bundles to $\calM_{\SL(2,\CC)}(X)$, such that the image of $\Theta$ is $\calI_q$. Denote by $\calM_{\ualpha}=\calM_{\ualpha}(\PP^1,D)$ the moduli space of semi-stable strongly parabolic $\SL(2,\CC)$-Higgs bundle $(\calE,\Phi)$ on $\PP^1$ with $D=0+1+\infty+y$ with parabolic weights $\ualpha=\left((\alpha_{y,1},\alpha_{y,2})_{y \in D}\right)$. Here~$\calE$ is a rank $2$ bundle of determinant $\det(\calE)=\calO(-4)$ together with complete flag $\{0 \} \subsetneq \calE_{y,1} \subsetneq \calE_{y,2}= \calE_y$ at each $y \in D$. By our convention, flags are ascending and weights descending, i.e., $\alpha_{y,1} > \alpha_{y,2}$. The Higgs field $\Phi \in H^0(\PP^1,\End_0(\calE)\otimes K(D))$ must preserve these flags, in the sense that $\Res(\Phi)(\calE_y) \subset \calE_{y,1}$. The Higgs bundle $(\calE,\Phi)$ is called semi-stable if and only if for each sub-Higgs bundle $(L,\psi) \subset (\calE,\Phi)$ we have $\pardeg(L,\psi)\leq \pardeg(\calE,\Phi)$. On $\bigl(\PP^1,D\bigr)$, this condition is automatically satisfied for all Higgs bundles that are not nilpotent. $\calM_{\ualpha}\bigl(\PP^1,D\bigr)$ carries a hyperk\"ahler metric defined by interpreting it as the moduli space of flat logarithmic connections with certain fix monodromy at $D$ via non-abelian Hodge theory. See~\cite{FMSW} for more details. For generic weights, the moduli space $\calM_{\underline{\alpha}}$ is isomorphic to the smooth surface $\calM_{\toy}$ constructed above.

Recall from Definition~\ref{defi:uniform_pillow_cover} that for a uniform pillowcase cover $(X,q)$ every fiber of $\check{p}\colon X \to \PP^1$ over one of the four marked points on $\PP^1$ has a well-defined ramification index.

\begin{theo}\label{theo:Theta}
 Let $(X,q)$ be a uniform pillowcase cover with simple zeros only. We define so-called compatible parabolic weights at $y \in D$ of ramification index $i$ by \smash{$\bigl(\alpha_0=\frac{i+1}{i+2}, \alpha_1=\frac{1}{i+2}\bigr)$}. Let $\calM_{\ualpha}\bigl(\PP^1,D\bigr)$ be the moduli space of strongly parabolic Higgs bundles on $\bigl(\PP^1,D\bigr)$ with compatible parabolic weights. Then there exists a holomorphic map
 \[ \Theta\colon \ \calM_{\ualpha}\bigl(\PP^1,D\bigr) \to \calM_{\SL(2,\CC)}(X)
 \] such that
 \begin{itemize}\itemsep=0pt
 \item[$(i)$] it maps the Hitchin section of $\calM_{\ualpha}$ to the Hitchin section of $\calM_{\SL(2,\CC)}(X)$ restricted to $\CC q$.
 \item[$(ii)$] for all $c \in \CC^\times$, it makes the following diagram commute
 \[ \begin{tikzcd} E \ar[r,"p^*"]\ar[d,"\cong"] & \Prym(\Sigma_q) \ar[d,"\cong"] \\
 \Hit^{-1}(c\eta) \ar[r,"\Theta"] & \Hit^{-1}(cq),
 \end{tikzcd} \]
 where we used the Hitchin section to identify the corresponding Hitchin fibers with abelian varieties.
 \item[$(iii)$] $\Theta$ can be promoted to a morphism of hermitian Higgs bundles, such that solutions to the Hitchin equation on \smash{$\bigl(\PP^1,D\bigr)$} are mapped to solutions to the Hitchin equation on~$X$.
 \end{itemize}
\end{theo}

\begin{proof}
 Recall the square of coverings of Figure~\ref{fig:square-of-covers}. The morphism $\Theta$ is given by a Hecke modified pullback along $\check{p}$. Let $(\calE,\Phi) \in \calM_{\ualpha}$. First define
 \[ (\calE',\Phi')= \bigl(\check{p}^*(\calE\otimes\calO(3))\otimes K_X^{-\frac12}, \check{p}^*\Phi\bigr).
 \]
 This is a meromorphic Higgs bundle on $X$ with $\tr(\Phi')=0$. By assumption of $q$ having simple zeros only all points in $\check{p}^{-1}(D)$ are ramification points of $\check{p}$, which are $2:1$ or $3:1$. Let $R_i \in \Div^+(X)$ be the divisor that has weight $1$ at the branch points that are $(i+1):1$ and $R=R_1 +R_2$. The pullback of the quasi-parabolic structure defines a quasi-parabolic structure on $\calE'$ at $R$ given by $\calE'_{x,1}:=\check{p}^*\calE_{\check{p}x,1} \subset \calE_{x}$. Now we define a Hecke modification
 \[ 0 \to \bigl(\hat{\calE},\hat{\Phi}\bigr) \to (\calE',\Phi') \to \bigoplus_{x \in \supp R} \calE'_{x}/\calE'_{x,1} \to 0.
 \]
 Then the map $\Theta$ is defined as
 $\Theta\colon (\calE,\Phi) \mapsto (\hat{\calE},\hat{\Phi})$.

 \begin{figure}[h]
$$
 \begin{tikzcd}
	& (\Sigma,\lambda) \ar[rd,"p"] \ar[ld,"\pi" above]& \\ (X,q) \ar[rd,"\check{p}"] & & (E, \omega) \ar[ld,"\check{\pi}"] \\ & \bigl(\PP^1,\eta\bigr) &
 \end{tikzcd}\vspace{-3mm}
$$
 \caption{Square of coverings associated to a pillowcase.}\label{fig:square-of-covers}
 \end{figure}
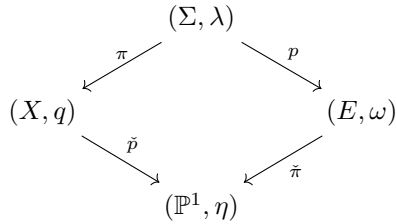

 \textit{Well-definedness:} First, we have to show that $\Theta$ is well defined. We have
 \[ \check{p}^*\calO_{\PP^1}(-2)=\check{p}^*K_{\PP^1}=K_X(-R_1-2R_2)
 \]
 and $\div(q)=R_2$. Hence, the determinant of $E'$ computes to
 \[ \det(\calE')=\check{p}^*\calO(2) \otimes K_X^{-1}=K_X^{-2}(R_1+2R_2)=\calO(R_1+R_2).
 \]
 Therefore, the determinant of $\hat{\calE}$ is $\det\bigl(\hat{\calE}\bigr)=\det(\calE')(-R_1-R_2)=\calO_X$.

 To show that the Higgs field $\hat{\Phi}$ is holomorphic we do a local computation. Locally at each $y \in D$, we can find a frame $s_1$, $s_2$ of $\calE$ adapted to the parabolic structure such that
 \begin{align} \Phi= \begin{pmatrix} \phi_0(z) \d z & \phi_1(z) \dfrac{\d z}{z} \vspace{1mm}\\ \phi_2(z) \d z & -\phi_0(z) \d z \end{pmatrix}.
 \label{eq:local-model-Phi}
 \end{align}
 and $\calE_{y,1}=\langle s_1\rest_{y} \rangle$. Hence, at all ramification points $x \in \check{p}^{-1}D$ the pullback has the form
 \begin{align} \label{eq:local-model-Phi'}
 \check{p}^* \Phi= \begin{pmatrix} \phi_0\bigl(w^k\bigr) w^{k-1} \d w & \phi_1\bigl(w^k\bigr) \dfrac{ \d w}{w} \vspace{1mm}\\ \phi_2\bigl(w^k\bigr)w^{k-1} \d w & -\phi_0\bigl(w^k\bigr) w^{k-1} \d w \end{pmatrix},
 \end{align}
 where $w$ is a coordinate centered at $x$, such that $\check{p}\colon w \mapsto z=w^k$. By assumption, $k=2,3$. The pullback quasi-parabolic structure is given by $\calE'_{x,1}=\langle \check{p}^*s_1 \rest_y\rangle$. Now, it is an easy computation to see that the Hecke modification modifies the Higgs field to
 \[ \hat{\Phi} = \begin{pmatrix} \phi_0\bigl(w^k\bigr) w^{k-1} \d w & \phi_1\bigl(w^k\bigr) \d w \\[1mm] \phi_2\bigl(w^k\bigr)w^{k-2} \d w & -\phi_0\bigl(w^k\bigr) w^{k-1} \d w \end{pmatrix}.
 \]
 Hence, indeed $(\hat{\calE},\hat{\Phi})$ defines a $\SL(2,\CC)$-Higgs bundle. Its polystability will follow from the existence of a solution to Hitchin's equation at the end of the proof.

 \textit{Hitchin sections:} Now, we want to apply this morphism to a point in the Hitchin section. We~identify the Hitchin base of $\calM_{\ualpha}$ as $\{ c \eta \mid c \in \CC\}$ for a fixed quadratic differential $\eta$ with simple poles at~$D$. Denote by $q=\check{p}^*\eta$ its pullback. Then a point $(\calE,\Phi)$ in the Hitchin section of $\calM_{\ualpha}$ is given by
 \[ (\calE,\Phi)=\left( \calO(-1) \oplus \calO(-3), \begin{pmatrix} 0 & c \\ \eta & 0 \end{pmatrix} \right) \in \calM_{\ualpha}.
 \]
 After pullback and tensoring, we obtain
 \[ (\calE',\Phi')=\left( K^{\frac12}(R_1+R_2) \oplus K^{-\frac12}, \begin{pmatrix} 0 & c \check{p}^*1 \\ \check{p}^* \eta & 0 \end{pmatrix} \right).
 \]
 Here $\check{p}^*1 \in \check{p}^*\calO(2) \otimes \check{p}^*K_{\PP^1}$ has a zero of order $1$ at each $2:1$ ramification point and a zero of order $2$ at each $3:1$ ramification point. The pullback quasi-parabolic structure at $\check{p}^{-1}D$ is given by the second coordinate with respect to the splitting. Hence, the Hecke modification yields
 \[ \Theta(\calE,\Phi)=\left( K^{\frac12} \oplus K^{-\frac12}, \begin{pmatrix} 0 & cq \\ 1 & 0 \end{pmatrix} \right)=s_H(cq)
 \]
 as asserted.

 \textit{Compatibility with pullback by $p$:} We show that this map extends the spectral correspondence on the regular locus. First we showed above that the Hitchin section of $\calM_{\underline{\alpha}}$ is mapped on~the Hitchin section of $\calM_{\SL(2,\CC)}(X)$. It is easy to see that the eigen-line bundle on $E$ of the Hitchin section of $\calM_{\underline{\alpha}}$ is $\check{\pi}^*\calO(-3)$. Similarly, the eigen-line bundle on $\Sigma$ of the Hitchin section of~$\calM_{\SL(2,\CC)}(X)$ is~\smash{$\pi^*K_X^{-\frac12}$}. We use these two Hitchin sections to identify the fibers with the corresponding abelian varieties. Then $(\calE,\Phi) \in \calM_{\ualpha}$ corresponds to an element $L_1 \in \Jac(E)$. Similarly, $\bigl(\hat{\calE},\hat{\Phi}\bigr)$ corresponds to an element $L_2 \in \Prym(\Sigma_q)$. We need to show that $L_2=p^*L_1$. From the spectral correspondence, we have an exact sequence
 \[ 0 \to L_1 \otimes \check{\pi}^*\calO(-3) \to \check{\pi}^*\calE \xrightarrow{\check{\pi}^*\Phi-\omega \id_{\check{\pi}^* \calE}} \check{\pi}^* (\calE \otimes \calO(2) ).
 \]
 Tensoring with $\calO(3)$ and pulling back, we obtain
 \[ 0 \to p^*L_1 \to \pi^*\bigl(\calE' \otimes K_X^{\frac12}\bigr) \xrightarrow{\pi^*\Phi'-\lambda \id_{\check{\pi}^* \calE}} \pi^*\bigl(\calE' \otimes K_X^{\frac12}\bigr) \otimes (\check{\pi} \circ p)^*\calO(2).
 \]
 Here we used the commutativity of diagram in Figure~\ref{fig:square-of-covers}.
 Finally, twisting by \smash{$K_X^{-\frac12}$} we see that the line bundle associated to $\pi^*(\calE',\Phi')$ through the spectral correspondence is $p^*L_1$. The (pullback of the) Hecke modification $\hat{\calE} \to \calE$ can potentially change this line bundle by twisting it with a~divisor supported at~$\pi^{-1}R$. However, on the regular locus of $\calM_{\underline{\alpha}}$ the quasi-parabolic structure is uniquely determined through the Higgs field at each $p \in D$ (see \cite[Proposition~8.1]{FMSW}). Hence, we can compute the eigen-line bundle and the quasi-parabolic structure of $\pi^*(\calE',\Phi')$ with respect to the pullback of the local frame of~\eqref{eq:local-model-Phi}. Then it is easy to see that the eigen-line bundle \smash{$p^*L_1 \otimes K_x^{-\frac12}$} descends to a subbundle of $\pi^*(\hat{\calE},\hat{\Phi})$. Hence $p^*L_1=L_2$.

 \textit{Solutions to Hitchin equation:} Finally, we will show that a hermitian metric $h$ on ${(\calE,\Phi) \in \calM_{\ualpha}}$ that solves the Hitchin equation is transformed to a solution to the Hitchin equation for $\bigl(\hat{\calE},\hat{\Phi}\bigr) \in \calM_{\SL(2,\CC)}(X)$. First there is a section of $\calO(4)$ with divisors $D$. Promoting this section to have norm $1$ defines a singular hermitian metric on~$\calO(4)$ locally given by $\vert z \vert^{-2}$ at $y \in D$. This induces a singular hermitian metric on $\calO(1)$ and hence on $\calO(3)$. The latter is given by \smash{$\vert z \vert ^{-\frac32}$} at $y \in D$ and will be denoted by~$h_{\calO(3)}$. Similarly, $q \in H^0\bigl(X,K_X^2\bigr)$ induces a singular hermitian metric \smash{$h_{-\frac12 K_X}$} on \smash{$K_X^{-\frac 12}$} that is smooth away from $Z(q)$ and given by \smash{$\vert w \vert^{\frac12}$} locally at $x \in Z(q)$. It is easy to see that singular hermitian metrics defined in this way are automatically flat.

 We need to extend the morphism $\Theta$ to hermitian Higgs bundles $(\calE,\Phi,h)$. In the first step, we use the singular hermitian metrics defined above to obtain a hermitian metric on $\calE'$ by
 \[ (\calE',h')=\pi^* ((E,h) \otimes (\calO(3),h_{\calO(3)}) ) \otimes \bigl(K_X^{-\frac12},h_{-\frac12 K_X}\bigr).
 \]
 This hermitian metric is holomorphic on $X\setminus R$. The hermitian metric~$h'$ pulls back to a hermitian metric $\hat{h}$ on $\calE$ through the Hecke modification $\hat{\calE} \to \calE$ a priori holomorphic only on $X \setminus R$.

 Now, we start with a polystable Higgs bundle $(\calE,\Phi) \in \calM_{\alpha}$ and let $h$ be a solution to Hitchin's equation. By the flatness of the hermitian metrics $h_{\calO(3)}$ and \smash{$h_{-\frac12 K_X}$}, the resulting hermitian metric~$h'$ on $\calE$ will still be a solution to Hitchin's equation wherever it is smooth. By definition, the Hecke modification $\hat{\calE} \to \calE$ is an isomorphism on $X \setminus R$ and hence the induced metric~$\hat{h}$ is a solution to Hitchin's equation on this locus. To show that it defines a solution to Hitchin's equation on~$X$, we are left with showing that it extends smoothly over~$R$.

 To do so, we compute the local description of $\hat{h}$ at $x \in R$. The metric $h$ is adapted to the parabolic structure. Hence, at $y \in D$ we can find a local frame $s_1$, $s_2$ of $\calE$ such that firstly
 \[ h= \begin{pmatrix} \vert z \vert^{2\alpha_1} & \\ & \vert z \vert^{2\alpha_2} \end{pmatrix},
 \]
 secondly the Higgs field is given by \eqref{eq:local-model-Phi} and thirdly the quasi-parabolic structure is the ascending flag $\langle s_1 \rangle \subset \langle s_1,s_2 \rangle$. We have to consider two cases depending on the ramification index of the fiber over $y \in D$. We will only give the details for $y \in D$ of ramification index~$2$, i.e., $\check{p}^{-1}y$ is made up from $3:1$ ramification points.

 In this case, the compatible parabolic weight are \smash{$\ualpha_y=\bigl(\frac34, \frac14\bigr)$}, so that \smash{$h=\diag\bigl( \vert z \vert^{\frac{4}{3}}, \vert z \vert^{\frac{2}{3}}\bigr)$}. Tensoring with $(\calO(3),h_{\calO(3)})$, pulling back and then tensoring with \smash{$\bigl(K_X^{-\frac12},h_{-\frac12 K}\bigr)$}, we obtain a~local description for $h'$ at $x \in Z(q)$ with respect to the induced frame
 \[ h'= \check{p}^*\diag\bigl( \vert z \vert^{\frac{4}{3}-\frac32}, \vert z \vert^{\frac{2}{3}-\frac32}\bigr) \vert w \vert^{\frac12}= \diag\bigl( 1, \vert w \vert^{-2} \bigr).
 \]
 Here $w$ is local coordinate at $x \in \check{p}^{-1}y$, such that $\check{p}\colon w \mapsto w^3=z$. Finally, with respect to the frame $s_1$, $s_2$ the Hecke modification $\hat{\calE} \to \calE$ is given by $\diag( 1, w)$ and hence the induced metric on $\hat{\calE}$ is indeed smooth at $x \in R$ of ramification index~$2$. The case of ramification index~$1$ follows along the same lines. Hence, $\hat{h}$ defines a smooth solution to the Hitchin equation for the Higgs bundle $\bigl(\hat{\calE},\hat{\Phi}\bigr)$.

 In particular, for a polystable Higgs bundle $(\calE,\Phi) \in \calM_{\ualpha}$ the image $\Theta(\calE,\Phi)$ is polystable and hence indeed $\Theta$ defines a map of moduli spaces
 \[ \Theta\colon \ \calM_{\ualpha} \to \calM_{\SL(2,\CC)}(X).
 \]
 This finishes the proof.
\end{proof}

\begin{rema} If $(X,q)$ is an uniform pillowcase cover such that there is an odd number of $y \in D$ of ramification index $1$, then the compatible weights are generic in the sense that semi-stability implies stability. In particular, $\calM_{\ualpha}$ is the elliptic surface referred to as Hausel's toy model with the nilpotent cone being of Kodaira type $I_0^*$.
\end{rema}

\begin{coro}\label{coro:hyperholomorphic} Let $(X,q)$ be a uniform pillowcase cover. Then $\calI_q \subset \calM_{\PGL(2,\CC)}(X)$ is a hyperholomorphic subvariety.
\end{coro}
\begin{proof} By Theorem~\ref{theo:Theta}, $\Theta$ maps solutions to the Hitchin equation on $\bigl(\PP^1,D\bigr)$ to solutions to the Hitchin equation on~$X$. Hence, it is holomorphic not only with respect to the holomorphic structure~$I$, but also with respect to the holomorphic structure~$J$ and~$K$ on the moduli spaces $\calM_{\ualpha}\bigl(\PP^1,D\bigr)$ and $\calM_{\SL(2,\CC)}(X)$. In particular, its image is a hyperholomorphic subvariety. However, Theorem~\ref{theo:Theta}\,(i) and~(ii) shows that $\Theta$ restricts to an isomorphism from the regular locus of $\calM_{\ualpha}$ to the torsor $\calI' \subset \calM_{\SL(2,\CC)}(X)$ obtained by acting with the abelian scheme $E \subset \Prym(\Sigma/\CC^\times q )$ over $\CC^\times q$ on the Hitchin section. Hence, $\overline{\calI}'$ the image of $\Theta$ is a hyperholomorphic subvariety. In particular, its image $\calI_q$ under the quotient map $\delta\colon \calM_{\SL(2,\CC)}(X) \to \calM_{\PGL(2,\CC)}(X)$ is hyperholomorphic.
\end{proof}

\subsection*{Acknowledgements} Thanks to Mirko Mauri, whose question initiated this work. Thanks to Martin M\"oller for his interest, many inspiring discussions as well as the suggestion to look at parallelogram-tiled surfaces. Thanks to Jochen Heinloth, Andr\'e Oliveira and Tom Sutherland for discussion. We~thank the~referees for their valuable comments and corrections. Research of the authors is supported by the Deutsche Forschungsgemeinschaft (DFG, German Research Foundation)~-- Project-ID~444845124~-- TRR~326.

\pdfbookmark[1]{References}{ref}
\LastPageEnding


\begin{thebibliography}{99}
\footnotesize\itemsep=0pt

\bibitem{Baraglia}
Baraglia D., Classification of the automorphism and isometry groups of {H}iggs
 bundle moduli spaces,
 \href{https://doi.org/10.1112/plms/pdw014}{\textit{Proc. Lond. Math. Soc.}}
 \textbf{112} (2016), 827--854,
 \href{http://arxiv.org/abs/1411.2228}{arXiv:1411.2228}.

\bibitem{BaragliaSchaposnik1}
Baraglia D., Schaposnik L.P., Real structures on moduli spaces of {H}iggs
 bundles, \href{https://doi.org/10.4310/ATMP.2016.v20.n3.a2}{\textit{Adv.
 Theor. Math. Phys.}} \textbf{20} (2016), 525--551,
 \href{http://arxiv.org/abs/1309.1195}{arXiv:1309.1195}.

\bibitem{BaragliaSchaposnik2}
Baraglia D., Schaposnik L.P., Cayley and {L}anglands type correspondences for
 orthogonal {H}iggs bundles,
 \href{https://doi.org/10.1090/tran/7587}{\textit{Trans. Amer. Math. Soc.}}
 \textbf{371} (2019), 7451--7492,
 \href{http://arxiv.org/abs/1708.08828}{arXiv:1708.08828}.

\bibitem{DonagiPantev}
Donagi R., Pantev T., Langlands duality for {H}itchin systems,
 \href{https://doi.org/10.1007/s00222-012-0373-8}{\textit{Invent. Math.}}
 \textbf{189} (2012), 653--735,
 \href{http://arxiv.org/abs/math.AG/0604617}{arXiv:math.AG/0604617}.

\bibitem{visible}
Evans J., Lectures on {L}agrangian torus fibrations, \textit{London Math. Soc.
 Stud. Texts}, Vol.~105,
 \href{https://doi.org/10.1017/9781009372671}{Cambridge University Press},
 Cambridge, 2023, \href{http://arxiv.org/abs/2110.08643}{arXiv:2110.08643}.

\bibitem{FGOP}
Franco E., Gothen P.B., Oliveira A., Pe\'{o}n-Nieto A., Unramified covers and
 branes on the {H}itchin system,
 \href{https://doi.org/10.1016/j.aim.2020.107493}{\textit{Adv. Math.}}
 \textbf{377} (2021), Paper No. 107493, 61,
 \href{http://arxiv.org/abs/1802.05237}{arXiv:1802.05237}.

\bibitem{FHHO}
Franco E., Hanson R., Horn J., Oliveira A., Lagrangians of Hecke cycles, {i}n
 preparation.

\bibitem{FrancoJardim}
Franco E., Jardim M., Mirror symmetry for {N}ahm branes,
 \href{https://doi.org/10.46298/epiga.2022.6604}{\textit{\'{E}pijournal
 G\'{e}om. Alg\'{e}brique}} \textbf{6} (2022), 4, 29~pages,
 \href{http://arxiv.org/abs/1709.01314}{arXiv:1709.01314}.

\bibitem{FrancoPeon}
Franco E., Pe\'{o}n-Nieto A., Branes on the singular locus of the {H}itchin
 system via {B}orel and other parabolic subgroups,
 \href{https://doi.org/10.1002/mana.202000267}{\textit{Math. Nachr.}}
 \textbf{296} (2023), 1803--1841,
 \href{http://arxiv.org/abs/1709.03549}{arXiv:1709.03549}.

\bibitem{FMSW}
Fredrickson L., Mazzeo R., Swoboda J., Weiss H., Asymptotic geometry of the
 moduli space of parabolic {${\rm SL}(2,{\mathbb C})$}-{H}iggs bundles,
 \href{https://doi.org/10.1112/jlms.12581}{\textit{J.~Lond. Math. Soc.}}
 \textbf{106} (2022), 590--661,
 \href{http://arxiv.org/abs/2001.03682}{arXiv:2001.03682}.

\bibitem{Freed}
Freed D.S., Special {K}\"{a}hler manifolds,
 \href{https://doi.org/10.1007/s002200050604}{\textit{Comm. Math. Phys.}}
 \textbf{203} (1999), 31--52,
 \href{http://arxiv.org/abs/hep-th/9712042}{arXiv:hep-th/9712042}.

\bibitem{GAP4}
{{GAP} -- {G}roups}, {A}lgorithms, and programming, {V}ersion~4.12.2, 2022,
 \url{https://www.gap-system.org}.

\bibitem{HauselToy}
Hausel T., Compactification of moduli of {H}iggs bundles,
 \href{https://doi.org/10.1515/crll.1998.096}{\textit{J.~Reine Angew. Math.}}
 \textbf{503} (1998), 169--192,
 \href{http://arxiv.org/abs/math.AG/9804083}{arXiv:math.AG/9804083}.

\bibitem{HauselHitchin}
Hausel T., Hitchin N., Very stable {H}iggs bundles, equivariant multiplicity
 and mirror symmetry,
 \href{https://doi.org/10.1007/s00222-021-01093-7}{\textit{Invent. Math.}}
 \textbf{228} (2022), 893--989,
 \href{http://arxiv.org/abs/2101.08583}{arXiv:2101.08583}.

\bibitem{HauselThaddeus}
Hausel T., Thaddeus M., Mirror symmetry, {L}anglands duality, and the {H}itchin
 system, \href{https://doi.org/10.1007/s00222-003-0286-7}{\textit{Invent.
 Math.}} \textbf{153} (2003), 197--229,
 \href{http://arxiv.org/abs/math.AG0205236}{arXiv:math.AG0205236}.

\bibitem{HellerSchaposnik}
Heller S., Schaposnik L.P., Branes through finite group actions,
 \href{https://doi.org/10.1016/j.geomphys.2018.03.014}{\textit{J.~Geom.
 Phys.}} \textbf{129} (2018), 279--293,
 \href{http://arxiv.org/abs/1611.00391}{arXiv:1611.00391}.

\bibitem{HitchinSpUmm}
Hitchin N., Higgs bundles and characteristic classes, in Arbeitstagung {B}onn
 2013, \textit{Progr. Math.}, Vol.~319,
 \href{https://doi.org/10.1007/978-3-319-43648-7_8}{Birkh\"{a}user}, Cham,
 2016, 247--264, \href{http://arxiv.org/abs/1308.4603}{arXiv:1308.4603}.

\bibitem{HitchinSpinors}
Hitchin N., Spinors, {L}agrangians and rank~2 {H}iggs bundles,
 \href{https://doi.org/10.1112/plms.12034}{\textit{Proc. Lond. Math. Soc.}}
 \textbf{115} (2017), 33--54,
 \href{http://arxiv.org/abs/1605.06385}{arXiv:1605.06385}.

\bibitem{Horn1}
Horn J., Semi-abelian spectral data for singular fibres of the
 {$\mathrm{SL}(2,\mathbb C)$}-{H}itchin system,
 \href{https://doi.org/10.1093/imrn/rnaa273}{\textit{Int. Math. Res. Not.}}
 \textbf{2022} (2022), 3860--3917,
 \href{http://arxiv.org/abs/2003.07806}{arXiv:2003.07806}.

\bibitem{KapustinWitten}
Kapustin A., Witten E., Electric-magnetic duality and the geometric {L}anglands
 program, \href{https://doi.org/10.4310/CNTP.2007.v1.n1.a1}{\textit{Commun.
 Number Theory Phys.}} \textbf{1} (2007), 1--236,
 \href{http://arxiv.org/abs/hep-th/0604151}{arXiv:hep-th/0604151}.

\bibitem{Markman}
Markman E., Spectral curves and integrable systems, \textit{Compositio Math.}
 \textbf{93} (1994), 255--290.

\bibitem{Mumford}
Mumford D., Prym varieties.~{I}, in Contributions to Analysis (a Collection of
 Papers Dedicated to {L}ipman {B}ers), Academic Press, New York, 1974,
 325--350.

\bibitem{PareschiPopa}
Pareschi G., Popa M., G{V}-sheaves, {F}ourier--{M}ukai transform, and generic
 vanishing, \href{https://doi.org/10.1353/ajm.2011.0000}{\textit{Amer.~J.
 Math.}} \textbf{133} (2011), 235--271,
 \href{http://arxiv.org/abs/math.AG/0608127}{arXiv:math.AG/0608127}.

\bibitem{PaulyPeon}
Pauly C., Pe\'{o}n-Nieto A., Very stable bundles and properness of the
 {H}itchin map, \href{https://doi.org/10.1007/s10711-018-0333-6}{\textit{Geom.
 Dedicata}} \textbf{198} (2019), 143--148,
 \href{http://arxiv.org/abs/1710.10152}{arXiv:1710.10152}.

\bibitem{Schnell}
Schnell C., The {F}ourier--{M}ukai transform made easy,
 \href{https://doi.org/10.4310/pamq.2022.v18.n4.a14}{\textit{Pure Appl.
 Math.~Q.}} \textbf{18} (2022), 1749--1770,
 \href{http://arxiv.org/abs/1905.13287}{arXiv:1905.13287}.

\end{thebibliography}
\end{document}